\documentclass[a4paper,12pt,reqno]{amsart}
\usepackage{latexsym,amscd,amssymb,url}
\usepackage{xcolor}
\definecolor{dblue}{rgb}{0,0,.6}
\usepackage[colorlinks=true, linkcolor=dblue, citecolor=dblue, filecolor = dblue, menucolor = dblue, urlcolor = dblue]{hyperref}

\usepackage[all,cmtip]{xy}  
\usepackage{graphicx}
\usepackage{color}
\usepackage{hyperref}
\usepackage{enumerate}

\newtheorem{theorem}{Theorem}[section]

\theoremstyle{plain}

\newtheorem{corollary}[theorem]{Corollary}

\newtheorem{example}[theorem]{Example}

\newtheorem{lemma}[theorem]{Lemma}

\newtheorem{remark}[theorem]{Remark}

\newcommand{\Z}{\mathbb Z}
\newcommand{\Q}{\mathbb Q}

\newcommand{\C}{\mathbb C}

\newcommand{\CP}{\mathbb P}

\newcommand{\im}{\operatorname{im}}

\newcommand{\Pic}{\operatorname{Pic}}

\newcommand{\Aut}{\operatorname{Aut}}

\newcommand{\id}{\operatorname{id}}

\newcommand{\GL}{\operatorname{GL}}

\newcommand{\pr}{\operatorname{pr}}

\newcommand{\Alb}{\operatorname{Alb}}

\newcommand{\Proj}{\operatorname{Proj}}

\newcommand{\dashedlongrightarrow}{\xymatrix@1@=15pt{\ar@{-->}[r]&}}
\renewcommand{\longrightarrow}{\xymatrix@1@=15pt{\ar[r]&}}
\renewcommand{\mapsto}{\xymatrix@1@=15pt{\ar@{|->}[r]&}}
\renewcommand{\twoheadrightarrow}{\xymatrix@1@=15pt{\ar@{->>}[r]&}}
\newcommand{\hooklongrightarrow}{\xymatrix@1@=15pt{\ar@{^(->}[r]&}}
\newcommand{\congpf}{\xymatrix@1@=15pt{\ar[r]^-\sim&}}
\renewcommand{\cong}{\simeq}

\begin{document}

\title[Equality in the Bogomolov--Miyaoka--Yau inequality]{Equality in the Bogomolov--Miyaoka--Yau inequality in the non-general type case} 

\author{Feng Hao}

\address{
KU Leuven, Celestijnenlaan 200B, B-3001 Leuven, Belgium 
}
\email{feng.hao@kuleuven.be}

\author{Stefan Schreieder}


\address{Institut f\"ur Algebraische Geometrie, Leibniz Universit\"at Hannover, Welfengarten 1, 30167 Hannover, Germany}
\email{schreieder@math.uni-hannover.de}

\date{December 23, 2020}
\subjclass[2010]{primary 14J27, 14E30; secondary 14J30, 32Q57} 
 

\keywords{Bogomolov--Miyaoka--Yau inequality, good minimal models, elliptic fibre spaces, classification theory, threefolds.}

\begin{abstract}   
We classify all minimal models  $X$ of dimension $n$, Kodaira dimension $n-1$ and with vanishing Chern number $c_1^{n-2}c_2(X)=0$.
This solves a problem of Koll\'ar.

Completing previous work of Koll\'ar and Grassi, we also show that there is a universal constant $\epsilon>0$ such that any minimal threefold  satisfies either $c_1c_2=0$ or $-c_1c_2>\epsilon$. 
This settles completely a conjecture of Koll\'ar.
\end{abstract}

\maketitle

\section{Introduction}

We work over the field of complex numbers.
A minimal model is a projective terminal $\Q$-factorial variety $X$ such that $K_X$ is nef.
By \cite{BCHM} and \cite{Lai11}, any smooth projective variety $\tilde X$ of dimension $n$ and Kodaira dimension $  \kappa(\tilde{X})\geq n-3$ admits a minimal model $X$ that is birational to $\tilde X$.

A minimal model $X$ is good if some multiple of $K_X$ is base-point free.
The abundance conjecture predicts that any minimal model is good.
By \cite{Ka92,Lai11}, this is known in arbitrary dimension $n$ if  $\kappa(X)\geq n-3$, see Section \ref{sec:kappa=n-1} below. 

If $X$ is an $n$-dimensional minimal model of general type (i.e.\ $\kappa(X)=n$), then we have the  
Bogomolov--Miyaoka--Yau inequality
\begin{align} \label{eq:miyaoka-ineq-gen-type}
(-1)^nc_1^{n-2}c_2(X)\geq (-1)^n\frac{n}{2n+2}c_1^n(X),
\end{align}
see \cite{Yau77} for the smooth case and \cite{GKPT} in general. 
Moreover, equality holds if and only if $X$ is birational to a (mildly singular) ball quotient, see \cite{Yau77,GKPT}. 

If $X$ is a good minimal model of dimension $n$ and Kodaira dimension $\kappa(X)\leq n-1$, then $c_1^n(X)=0$ and the Bogomolov--Miyaoka--Yau inequality simplifies to 
\begin{align} \label{eq:miyaoka-ineq}
(-1)^{n}c_1^{n-2}c_2(X)\geq 0 ,
\end{align} 
which holds by \cite{miyaoka}.
As above, it is natural to ask for a classification of all cases where equality holds in (\ref{eq:miyaoka-ineq}), i.e., $c_1^{n-2}c_2(X)=0$.  
This problem goes back to Koll\'ar, who posed it for threefolds of Kodaira dimension two, see \cite[Remark 3.6]{Kol94}. 
A related question has been studied by  Peternell and Wilson \cite{PW96}, who classify minimal Gorenstein threefolds $X$ where $c_1^2(X)-3c_2(X)$ is numerically trivial (by (\ref{eq:miyaoka-ineq-gen-type}) this implies $\kappa(X)\leq 2$ and the condition is much stronger than $c_1c_2(X)=0$).

Let now $X$ be a good minimal model with $\kappa(X)\leq n-2$.
It is easy to see that $c_1^{n-2}c_2(X)=0$ is automatic in the case where $\kappa(X)<n-2$.
If $\kappa(X)=n-2$, then the general fibre $F$ of the Iitaka fibration $f:X\to S$ is a smooth minimal surface of Kodaira dimenison zero, hence a K3, Enriques, abelian or bi-elliptic surface.
We will show (see Lemma \ref{lem:kappa<n-1}) that in this case
 $c_1^{n-2}c_2(X)=0$ is equivalent to asking that $F$ is an abelian or a bi-elliptic surface. 
 In other words, a good minimal model $X$ with $\kappa(X)\leq n-2$ satisfies $c_1^{n-2}c_2(X)=0$ if and only if the general fibre of the Iitaka fibration is neither a K3-surface nor an Enriques surface.
This simple observation reduces the classification problem of $n$-dimensional good minimal models with $c_1^{n-2}c_2(X)=0$ to the case of Kodaira dimension $n-1$. 
In this paper we solve this problem completely.

\begin{theorem} \label{thm:kappa=n-1} 
Let $X$ be a minimal model with $\dim X=n$ and $\kappa(X)=n-1$.
Then $c_1^{n-2}c_2(X)=0$ if and only if $X$ is birational to a quotient $Y=(T\times E)/G$, where
\begin{enumerate}
\item $Y$ has canonical singularities; \label{item:thm1:Y=sm-in-codim2}
\item $E$ is an elliptic curve and $T$ is a normal projective variety with  $K_T$ ample;
 \label{item:thm1:KT=ample}
\item $G\subset \Aut(T)\times \Aut(E)$ is a finite group which acts diagonally, faithfully on each factor and freely in codimension two  on $T\times E$. 
\label{item:thm1:Gacts-diag}
\end{enumerate}
\end{theorem}

The above theorem makes it very easy to construct minimal models $X$ with $\kappa(X)=n-1$ and $c_1^{n-2}c_2(X)=0$.
Indeed, we can  start with any finite subgroup $G\subset \Aut(E)$ of some elliptic curve $E$ and pick an $(n-1)$-dimensional  
canonical model $T$ 
with a faithful $G$-action such that the quotient $Y=(T\times E)/G$ has canonical singularities and such that $G$ acts freely in codimension two on $T\times E$.
Then any terminalization $X^+$ is a minimal model with $\kappa=n-1$ and $c_1^{n-2}c_2=0$.
Moreover, the same holds for any variety $X$ produced via a finite sequence of flops $X^+\dashrightarrow X$. 

Conversely, the following corollary shows that any minimal model $X$ with $\kappa(X)=n-1$ and $c_1^{n-2}c_2(X)=0$ arises this way.
Moreover,  the triple $(T,E,G)$ is uniquely determined by the birational type of the minimal model $X$, and so the above  correspondence is a one to one correspondence.

\begin{corollary} \label{cor:factor-phi}
The following properties 
 are consequences of items (\ref{item:thm1:Y=sm-in-codim2}), (\ref{item:thm1:KT=ample}) and (\ref{item:thm1:Gacts-diag}) in Theorem \ref{thm:kappa=n-1} and hence hold true in the  notation of Theorem \ref{thm:kappa=n-1}:
 \begin{enumerate}[(a)]
 \item $K_Y$ is nef and the birational map $X\dashrightarrow Y$ factors into a sequence of flops $X\dashrightarrow X^{+}$ and a terminalization $X^{+}\to Y$;\label{item:cor:factor-phi:1}
 \item  if $f:X\to S$ is the Iitaka fibration of $X$, then $S\cong T/G$ and $f$ coincides with the natural composition $X\dashrightarrow Y\to T/G$;\label{item:cor:factor-phi:2}
  \item $T$ has canonical singularities and so (as $K_T$ is ample) $T$ is its own canonical model;\label{item:cor:factor-phi:3} 
 \item the varieties $T$ and $E$ as well as the subgroup $G\subset \Aut(T)\times \Aut(E)$ (hence also the quotient $Y=(T\times E)/G$) are uniquely  determined up to isomorphism by the birational equivalence class of $X$.\label{item:cor:factor-phi:4} 
\end{enumerate}    
\end{corollary}
 
 By item (\ref{item:cor:factor-phi:4}), the model $Y=(T\times E)/G$ is unique in its birational equivalence class. 
Moreover, item (\ref{item:cor:factor-phi:1}) in Corollary \ref{cor:factor-phi} implies that any minimal model birational to $Y$ is obtained by the composition of a terminalization and a sequence of flops. 
This is in complete analogy with the properties of the canonical model of a variety of general type, and so we may think about $Y$ in Theorem \ref{thm:kappa=n-1} as a suitable replacement of canonical models in our situation. 

The assertion in item (\ref{item:thm1:Gacts-diag}) of Theorem \ref{thm:kappa=n-1}, claiming that the group $G$  acts diagonally on $T\times E$, is crucial.
While diagonal group actions are clearly easier to handle and form a much smaller class than arbitrary group actions, our proof of Theorem \ref{thm:kappa=n-1} depends heavily on the insight that we can arrange the group action to be diagonal.
This is essentially the content of Theorem \ref{thm:quotient} below, which is a key result that allows us to pass from arbitrary group actions on $T\times E$ to diagonal ones.  
The result is nontrivial and somewhat surprising.
In fact, already in the case of surfaces, it is easy to construct examples along the following lines, see Section 5.2 below for detailed constructions. 

\begin{example} \label{ex:CxE}
There is a minimal projective surface $X$ with $\kappa(X)=1$ and $c_2(X)=0$, such that $X=(C\times E)/G$, where $C$ is a smooth projective curve of genus at least two, $E$ is an elliptic curve and $G\subset \Aut(C\times E)$ is cyclic of order three, such that the action of $G$ on $C\times E$ is free but not diagonal, i.e.\ $ G \not\subset \Aut(C)\times \Aut(E)$.
\end{example}

In sharp contrast, Theorem \ref{thm:kappa=n-1} and Corollary \ref{cor:factor-phi} show that any minimal projective surface $X$ with $\kappa(X)=1$ and $c_2(X)=0$ is isomorphic to the quotient of a product $C'\times E'$ of some smooth projective curve $C'$ of genus at least two and an elliptic curve $E'$ by a finite group action that is diagonal and free. 
While it is classically known that bielliptic surfaces can be expressed as quotients of two elliptic curves by a finite group action that is diagonal and free, this analogous statement for minimal surfaces of Kodaira dimension one and vanishing Euler number seems new. 

The fact that $G$ acts diagonally 
 also restricts the possible groups  that can appear in Theorem \ref{thm:kappa=n-1}.
Indeed, $G$ must be a finite subgroup of the automorphism group of an elliptic curve and so it must be an extension $1\to H\to G\to \Z/m\to 1$ of a finite cyclic group of order $m\leq 4$ or $m=6$ by a finite subgroup $H\subset (\Q/\Z)^2$. 
In fact, we will prove also the converse to this observation, giving rise to a complete classification of all groups that appear in Theorem \ref{thm:kappa=n-1}. 

 \begin{corollary} \label{cor:classification-of-groups}
 There is a minimal model $X$ of some dimension $n\geq 2$ and birational to $Y=(T\times E)/G$ as in Theorem \ref{thm:kappa=n-1} if and only if $G$ is a finite subgroup of the automorphism group of an elliptic curve.
 \end{corollary} 

For a minimal model $X$, $c_1^{n-2}c_2(X)$ is a priori only a rational number, because $K_X$ is only $\Q$-Cartier.
The second main result of this paper shows that for minimal threefolds, this number, if nonzero,  is universally bounded away from zero.
This completes work of Koll\'ar \cite{Kol94} and Grassi \cite{gra94} and solves completely a conjecture of Koll\'ar, see 
\cite[Conjecture 3.5]{Kol94}.

\begin{theorem} \label{thm:c1c2>epsilon}
There is a positive constant $\epsilon>0$, such that for any minimal model $X$ of dimension three, we have either $c_1c_2(X)=0$ or $-c_1c_2(X)\geq\epsilon$.
\end{theorem}

By work of Koll\'ar \cite{Kol94} and Grassi \cite{gra94}, and the boundedness of threefolds of general type and bounded volume \cite{HM06,Tak06}, Theorem \ref{thm:c1c2>epsilon} reduces to the case where $X$ has Kodaira dimension two and  the Iitaka fibration $X\to S$ is generically isotrivial. 
In this paper we will settle this remaining case.

Theorems \ref{thm:kappa=n-1} and \ref{thm:c1c2>epsilon} will both be deduced from the following 
theorem, which classifies all minimal models of dimension $n$ and Kodaira dimension $n-1$ whose Iitaka fibration is generically isotrivial. 

\begin{theorem} \label{thm:quotient-type}
For a minimal model $X$ of dimension $n$ and Kodaira dimension $n-1$, the following are equivalent.
\begin{enumerate}
\item The Iitaka fibration $f:X\to S$ is generically isotrivial, i.e.\ the fibres of $f$ over a dense open subset of $S$ are elliptic curves with constant $j$-invariants.\label{item:quotient-type:1}
\item  \label{item:quotient-type:2}
There is an elliptic curve $E$, a normal projective variety $T$ with ample canonical class $K_T$, and a finite subgroup $G\subset \Aut(T)\times \Aut(E)$ with the following properties.
The action of $G$ on $T\times E$ is diagonal and faithful on each factor, such that the quotient $Y:=(T\times E)/G$ has only canonical singularities. 
Moreover, there is a commutative diagram 
$$
\xymatrix{
X \ar[dr]_f \ar@{-->}[r]^-\phi &X^+\ar[r]^-\tau & Y=(T\times E)/G \ar[dl]^{g}\\
& S  \cong T/G & ,
}
$$
where $f$ is the Iitaka fibration of $X$, $\phi$ is a composition of flops, $\tau$ is a terminalization, and $g$ is induced by the projection $T\times E\to T$. 
Moreover,  there is a normal subgroup $G_0\subset G$ of index $\leq 4$ or $6$ whose action on $T\times E$ is free. 
\end{enumerate}
\end{theorem}

Again, the fact that $G$ acts diagonally (which relies on the aforementioned Theorem \ref{thm:quotient} below) is not only a convenient statement, but it is also essential for our proof (e.g.\ for the existence of $\phi$ and $\tau$). 
The main point is that a diagonal action of a finite group $G$ on $T\times E$ that is faithful on each factor is automatically free in codimension one and so $T\times E\to (T\times E)/G$ is quasi-\'etale. 
This property will be crucial, as it allows to translate between the birational geometry of $T\times E$ and that of the quotient  $Y=(T\times E)/G$ in an effective way (e.g.\ $K_{T\times E}$ is nef if and only if $K_Y$ nef).

In Theorems \ref{thm:kappa=n-1} and \ref{thm:quotient-type}, the projective variety $T$ has automatically canonical singularities, see Corollary \ref{cor:T=canonical}.
Since canonical surfaces are Gorenstein, the assertion about the normal subgroup $G_0\subset G$ in Theorem \ref{thm:quotient-type} will lead us to the following qualitative statement in dimension three.

\begin{corollary} \label{cor:Cartier-index}
Let $X$ be a minimal model of dimension three and of Kodaira dimension two.
Assume that the Iitaka fibration $f:X\to S$ is generically isotrivial.

Then the Cartier index of $X$ is $\leq 4$ or $6$.
In particular, $12K_X$ is Cartier.
\end{corollary}
In particular, for any minimal model $X$ of dimension three and of Kodaira dimension two whose Iitaka fibration is generically isotrivial, $4c_1c_2(X)$ or $6c_1c_2(X)$ is an integer.
This implies that one can take $\epsilon=\frac{1}{6}$ in Theorem 1.5 in the remaining case of Kodaira dimension 2 and generically trivial
Iitaka fibration.  

Note also that the Iitaka fibration is generically isotrivial if $c_1c_2(X)=0$ (see Lemma \ref{lem:grassi} below) and so the above corollary bounds the Cartier index of $K_X$ in this situation. 
  
\section{Preliminaries}
\subsection{Conventions and notation}\label{subsec:convention}
We work over the field of complex numbers.
A variety is an integral separated scheme of finite type over $\C$.
A minimal model is a projective variety $X$ with terminal $\Q$-factorial singularities such that $K_X$ is nef.

An open subset of a variety is big if its complement has codimension at least two.

Linear equivalence of divisors is denoted by $\sim$ and $\Q$-linear equivalence by $\sim_\Q$.
In particular, $D_1 \sim_{\Q} D_2$ if and only if $mD_1\sim mD_2$ for some positive integer $m$.

For a birational map $\varphi:X\dashrightarrow Y$ between projective varieties $X$ and $Y$, we denote for any $\Q$-Cartier $\Q$-divisor $D$ on $X$ by $\varphi_\ast D$ the $\Q$-divisor on $Y$ that is obtained by pulling back $D$ to a common resolution and pushing that pullback down to $Y$.
For a $\Q$-Cartier $\Q$-divisor $D'$ on $Y$, the pullback $\varphi^\ast D'$ is defined similarly and coincides with $\varphi^{-1}_\ast D'$ that we have just defined.

An elliptic fibre space is a normal quasi-projective variety $X$ with a pojective morphism $f:X\to S$ to a normal quasi-projective variety $S$ whose general fibre is an elliptic curve.
We say that $f$ has trivial monodromy, if $R^1f_\ast \Q$ restricts to a trivial local system over some non-empty Zariski open subset $U\subset S$.  

If $X$ is a variety which is smooth in codimension two (e.g.\ terminal), then $c_2(X)$ denotes the codimension two cycle on $X$, given by the closure of the second Chern class of the tangent bundle of the smooth locus of $X$.
If additionally $K_X$ is $\Q$-Cartier, then the intersection numbers $c_1^n(X):=(-K_X)^n$ and $c_1^{n-2}c_2(X):=(-K_X)^{n-2}c_2(X)$ are well-defined rational numbers, where $n=\dim X$.

\subsection{Terminalizations} \label{subsec:terminalization}

A proper birational morphism $\tau:Y'\to Y$ between normal varieties with $K_Y$ $\Q$-Cartier is a crepant birational contraction if $\tau^\ast K_{Y}=K_{Y'}$.
A terminalization of a variety $Y$ with canonical singularities is a crepant contraction $\tau:Y'\to Y$ from a $\Q$-factorial and terminal variety $Y'$. 
If $Y$ is canonical, then a terminalization $Y'$ of $Y$ exists by \cite{BCHM}.
Explicitly, $Y'$ is constructed by taking a resolution $\tilde Y\to Y$ and running a relative minimal model program of $\tilde Y$ over $Y$.

\subsection{$G$-equivariant minimal model program} \label{subsec:G-MMP}

Let $G$ be a finite group.
A variety with a $G$-action is called a $G$-variety.

Let $f:T\to S$ be a $G$-equivariant morphism of projective $G$-varieties with terminal singularities.
Assume that $K_T$ is $f$-big.
Then, by \cite{BCHM}, there is a unique relative canonical model $f^c:T^c\to S$, where $T^c$ has canonical singularities and $K_{T^c}$ is $f^c$-ample.
Explicitly, $T^c=\Proj (\bigoplus_{n\geq 0} f_\ast \omega_T^{\otimes n})$ and so $G$ acts on $T^c$, making  $f^c$ $G$-equivariant.

\subsection{Good minimal models for $\kappa\geq n-3$} \label{sec:kappa=n-1}
By \cite{BCHM} any smooth projective variety $X$ of dimension $n$ and Kodaira dimension $\kappa(X)=n$ admits a minimal model  $X^{min}$ that is birational to $X$.
Moreover, the basepoint free theorem implies that any such minimal model is good, see \cite[Theorem 3.3]{kollar-mori}.
By the  minimal model program in dimension three and a result of Lai \cite{Lai11}, the same result holds true if $\kappa(X)\geq n-3$.

\begin{theorem} \label{thm:good-minmod-kappa=n-1}
Let $X$ be a smooth projective variety of dimension $n$ and  $\kappa(X)\geq n-3$.
Then there is a minimal model $X^{min}$ that is birational to $X$.
Moreover, any such minimal model is good, i.e.\  for some integer $m>0$ the linear system $|mK_{X^{min}}|$ is base point free.
\end{theorem}
\begin{proof}
If $n=3$, then the result follows from the fact that the full minimal model program (including the abundance conjecture) is known for threefolds, see \cite{kollar-mori} and \cite{Ka92}.
The general result is therefore a direct consequence of \cite[Proposition 2.4 and Theorem 4.4]{Lai11}.
\end{proof}

\subsection{Good minimal models with $c_1^{n-2}c_2=0$ and $\kappa\leq n-2$}

The following well-known lemma essentially reduces the classification of all good minimal models with $c_1^{n-2}c_2=0$ to the case $\kappa=n-1$, c.f.\ \cite{gra94}.

\begin{lemma}\label{lem:kappa<n-1}
Let $X$ be a good minimal model of dimension $n$ and with Iitaka fibration $f:X\to S$.
If $\kappa(X)\leq n-2$, then $c_1^{n-2}c_2(X)=0$ is automatic if $\kappa(X)\leq n-3$ and it is equivalent to asking that the general fibre of $f$ is an abelian or a bielliptic surface if $\kappa(X)=n-2$. 
\end{lemma}
\begin{proof}
Since $f$ is the Iitaka fibration, $c_1(X)=-f^\ast A$ for an ample $\Q$-divisor $A$ on $S$ and so $c_1^{n-2}c_2(X)=0$ is automatic if $\kappa(X) \leq n-3$, because $\kappa(X) =\dim S$.
Moreover, if $\kappa(X)=n-2$, then $c_1^{n-2}c_2(X)$ is a nonzero multiple of $c_2(F)$ for a general fibre $F$ of $f$, because $F$ has trivial normal bundle and so $c_2(X)\cdot F=c_2(X)|_F=c_2(F)$.
Hence, $c_1^{n-2}c_2(X)=0$ if and only if $c_2(F)=0$.
Since $F$ has numerically trivial canonical bundle (given by the restriction of $K_X$), the condition $c_2(F)=0$ means that $F$ is an abelian or a bi-elliptic surface, see e.g.\ \cite[Chapter VI, Theorem (1.1)]{BHPV}.
\end{proof}

\subsection{\'Etale and quasi-\'etale morphisms}
A morphism $f:X\to Y$ between varieties is \'etale if it is flat and unramified.
Since we are working over the algebraically closed field $\C$, $f$ is \'etale if and only if for all $x\in X$, the induced morphism between the completed local rings $\widehat {\mathcal O}_{Y,f(x)} \to \widehat {\mathcal O}_{X,x}$ is an isomorphism for all $x\in X$.
If $f$ is finite, then $f$ is \'etale if and only if it is a topological covering of the underlying analytic spaces and in this situation, $f$ is uniquely determined by the finite index subgroup $f_\ast \pi_1(X)\subset \pi_1(Y)$.
Conversely, any finite index subgroup of $\pi_1(Y)$ corresponds uniquely to a finite \'etale covering of $Y$ as above. 

A finite  morphism $f:X\to Y$ between normal varieties is quasi-\'etale, if it is \'etale in codimension one, i.e.\ there is a big open subset $U\subset X$, such that $f|_U:U\to Y$ is \'etale, see \cite[Definition 3.3]{GKP16}.
In this situation, $f$ is automatically \'etale over the smooth locus of $Y$, see Corollary \ref{cor:quasi-etale=etale} below.

\subsection{Galois covers}
A finite morphism $f:X\to Y$ between quasi-projective varieties is Galois (or a Galois cover) if there is a finite subgroup $G\subset \Aut(X)$ such that $Y\cong X/G$ and $f$ is isomorphic to the quotient map $X\to X/G$.
The group $G$ is called the Galois group of $f$.
We will need the following consequence of an equivariant version of Zariski's main theorem, see \cite[Theorem 3.8]{GKP16}.

\begin{theorem} \label{thm:Zariski}
Let $f:U\to V$ be a finite morphism between normal quasi-projective varieties.
Let $Y$ be a normal projective closure of $V$.
Then there is a normal projective closure $X$ of $U$, which is unique up to unique isomorphism, such that $f$ extends to a finite morphism $\overline f:X \to Y$.
Moreover, if $f$ is Galois with group $G$, then $\overline f$ is also Galois with group $G$.
\end{theorem}

The above theorem implies for instance the following two well-known statements.

\begin{corollary} \label{cor:quasi-etale=etale}
Let $f:X\to Y$ be a finite quasi-\'etale morphism between normal quasi-projective varieties.
If $Y$ is smooth, then $f$ is \'etale.
\end{corollary}
\begin{proof}
By assumption, there is a big open subset $V\subset Y$ such that $U:=f^{-1}(V)\to V$ is a finite \'etale morphism.
Since $Y$ is smooth, $V$ is smooth and so $U$ must be smooth as well.
But finite \'etale morphisms $U\to V$ between smooth quasi-projective varieties are in one to one correspondence to finite index subgroups of $\pi_1(V)$.
Since $Y$ is smooth and $V\subset Y$ is big, $\pi_1(Y)\cong\pi_1(V)$ and so the finite \'etale cover $U\to V$ extends uniquely to a finite \'etale cover $\overline U\to Y$.
By Theorem \ref{thm:Zariski}, $\overline U\to Y$ and $X\to Y$ extend to finite covers of a normal projective closure of $Y$, which by construction coincide over $V\subset Y$.
The uniqueness assertion in Theorem \ref{thm:Zariski} thus implies that these extensions are isomorphic, and so $\overline U\to Y$ and $X\to Y$ must be isomorphic.
Hence, $X\to Y$ is \'etale, as we want. 
\end{proof}

\begin{corollary} \label{cor:Zariski}
Let $f:X\to Y$ and $g:Y\to Z$ be finite morphisms of normal quasi-projective varieties.
If $g\circ f:X\to Z$ is Galois, then so is $f$.
\end{corollary}
\begin{proof}
Since $f$ and $g$ are finite and $g\circ f$ is surjective, $f$ is finite and surjective.
Hence, $\bar{f}^{-1}(Y)=X$ for any extension $\bar f:\bar X\to \bar Y$ of $f$ to a finite morphism between projective closures of $X$ and $Y$.
By Theorem \ref{thm:Zariski}, $f$ is thus Galois if and only if its base change to a non-empty Zariski open subset $V\subset Y$ is Galois.
Hence, up to replacing $X,Y$ and $Z$ by suitable dense open subsets, we may assume that $X,Y$  and $Z$ are smooth and $f$ and $g$ are finite \'etale.
In this situation, consider the injective morphisms on fundamental groups
$$
\pi_1(X)\stackrel{f_\ast}\longrightarrow \pi_1(Y)\stackrel{g_\ast}\longrightarrow \pi_1(Z) .
$$
Since $g\circ f$ is a finite Galois \'etale cover,
$\im  (g_\ast \circ f_\ast)\subset \pi_1(Z)$ is a normal subgroup.
This implies that $\im  (g_\ast \circ f_\ast)\subset \im(g_\ast)$  is a normal subgroup of $\im(g_\ast)$.
Since $g_\ast$ is injective, we conclude that $\im(f_\ast)\subset  \pi_1(Y)$ is a normal subgroup, which is equivalent to saying that the finite \'etale morphism $f$ is Galois.
This proves the corollary.
\end{proof}

\subsection{Elliptic fibre spaces with $c_1^{n-2}c_2=0$} \label{subsec:Grassi} 



Let $X$ be an $n$-dimensional minimal model of Kodaira dimension $n-1$ and with Iitaka fibration $f:X\to S$, which is a morphism by Theorem \ref{thm:good-minmod-kappa=n-1}. 
Let $A$ be a very ample divisor on $S$ such that $K_X$ is linearly equivalent to a rational multiple of $f^\ast A$ and let $C\subset S$ be the intersection of $n-1$ general elements of the linear series $|A|$.
Since $S$ is normal and $X$ is terminal, it follows from Bertini's theorem that $Z:=f^{-1}(C)$ and $C$ are smooth.
We then have the following simple and well-known lemma.

\begin{lemma} \label{lem:grassi}
In the above notation, the following are equivalent:
\begin{enumerate}
\item $c_1^{n-2}c_2(X)=0$; \label{item:c1c2X=0}
\item $c_2(Z)=0$; \label{item:c2Z=0}
\item $Z\to C$ is a minimal elliptic surface such that all singular fibres are multiples of smooth elliptic curves;\label{item:Z=ell-surf1}
\item $Z\to C$ is a minimal elliptic surface, whose smooth fibres have constant $j$-invariants and whose singular fibres are multiples of smooth elliptic curves.\label{item:Z=ell-surf2}
\end{enumerate}   
\end{lemma}
\begin{proof}
%
Note first that the normal bundle   $\mathcal{N}_{Z/X}=f^*A^{\oplus n-2}|_Z$ is a direct sum of nef line bundles on $Z$. 
Since $K_X$ is nef, it follows that  $K_Z=K_X|_{Z}\otimes \overset{n-2}{\wedge}\mathcal{N}_{Z/X}$  is nef as well.
Hence, $f|_Z:Z\to C$ is a minimal elliptic surface of Kodaira dimension one.

Since $Z$ is smooth and contained in the smooth locus of $X$, we have a short exact sequence of vector bundles on $Z$:
$$
0\to T_Z\to T_X|_Z\to f^\ast A^{\oplus n-2}|_{Z}\to 0 .
$$ 
Applying the Whitney sum formula, we deduce that the second Chern number of $Z$ is given by
\begin{align*}
c_2(Z)&=c_2(X)\cdot f^\ast A^{n-2}-c_1(Z)\binom{n-2}{1}f^\ast A|_Z - \binom {n-2}{2}f^\ast A^2|_Z\\
&=c_2(X)\cdot f^\ast A^{n-2}-(n-2)c_1(Z) f^\ast A|_Z ,
\end{align*}
where  we used $f^\ast A^2|_Z=f^\ast A^n=0$.
By adjunction, $c_1(Z)=(c_1(X)- (n-2)f^\ast A)|_Z$, and so  
 $
c_1(Z) f^\ast A|_Z=0
$,
as it is a multiple of $f^\ast A^{n}=0$.
The above formula thus yields
$$
c_2(Z)=c_2(X)\cdot f^\ast A^{n-2} .
$$
This proves the equivalence of (\ref{item:c1c2X=0}) and (\ref{item:c2Z=0}), because $K_X$ is linearly equivalent to a nonzero rational multiple of $f^\ast A$.

To prove (\ref{item:c2Z=0}) $\Leftrightarrow $ (\ref{item:Z=ell-surf1}), note that  by \cite[Lemma VI.4]{Bea96}, 
 $c_2(Z)$  coincides with the sum of the Euler numbers of the singular fibres of $Z\to C$.
By Kodaira's classification of singular fibres it thus follows that $c_2(Z)=0$ if and only if all singular fibres of $Z\to C$ are multiples of a smooth elliptic curve, cf.\  \cite[Section V.7]{BHPV} for Kodaira's table of singular fibres of minimal elliptic fibrations.

Finally, (\ref{item:Z=ell-surf2}) $\Rightarrow $ (\ref{item:Z=ell-surf1}) is clear. 
For the converse, by $(\ref{item:Z=ell-surf1})$, the only singular fibres of the elliptic fibration $f|_Z: Z\to C$ are multiples of smooth elliptic curves. 
Then there is a ramified Galois cover $q: C'\to C$, a smooth projective surface $Z'$ and the following commutative diagram
$$
\xymatrix{
Z' \ar[d]_g \ar[r]  & Z\ar[d]^-{f|_Z}\\
C' \ar[r]^q & C,
}
$$ where $g$ is smooth. 
Also, the above diagram is Cartesian over a dense open set of $C$, see e.g., \cite[Lemma VI. 7, 7']{Bea96}. 
Notice that since $g$ is smooth, the $j$-invariant map $j: C'\to\CP^1$ is constant, because poles of $j$ correspond to singular fibres. 
Hence the j-invariant $j:C\dashrightarrow \CP^1$ is constant. 
This proves the lemma. 
\end{proof} 

The above lemma has the following immediate consequence, cf.\ \cite[Theorem 2.7]{gra94}. 

\begin{corollary} \label{cor:grassi}
Let $X$ be an $n$-dimensional minimal model of Kodaira dimension $n-1$ and with Iitaka fibration $f:X\to S$.
Then, $c_1^{n-2}c_2(X)=0$ if and only if $f$ is generically isotrivial and for any codimension one point $s\in S^{(1)}$, the fibre of $f$ above $s$ is either smooth or a multiple of a smooth elliptic curve.
\end{corollary}

\begin{proof}
Let  $C\subset S$ be a general complete intersection curve with preimage $Z:=f^{-1}(C)$.
Since $X$ is minimal and $f:X\to S$ is the Iitaka fibration, $Z\to C$ is a smooth minimal elliptic surface.
By Lemma \ref{lem:grassi}, $c_1^{n-2}c_2(X)=0$ if and only if $Z\to C$ is generically isotrivial with only multiple singular fibres.
The latter is equivalent to asking that $f$ is generically isotrivial and for any codimension one point $s\in S^{(1)}$, the fibre of $f$ above $s$ is either smooth or a multiple of a smooth elliptic curve. 
\end{proof}


\section{Reparametrizing quotients of products with elliptic curves}

In this section we analyse quotients $(T\times E)/G$, where $T$ is a smooth  projective variety, $E$ is an elliptic curve and $G$ is a finite group that acts faithfully on $T\times E$.
We will additionally assume that $G$ acts on $T$ such that the projection $T\times E\to T$ is $G$-equivariant.
The simplest such actions are diagonal, i.e.\ $G\subset \Aut(T)\times \Aut(E)$ acts separately  on $T$ and $E$, respectively.
Diagonal actions have the nice property that they are automatically free in codimension one, as long as they are faithful on each factor.
However, not every action as above needs to be diagonal.
An easy counterexample is given by the action of $\Z/2$ on the self-product $E\times E$ of an elliptic curve $E$, generated by $(x,y)\mapsto (-x,x+y)$. 
This action is not diagonal, nor free in codimension one, and it is easy to construct many more examples along these lines (also for $T$ of general type), see Section \ref{subsec:example} below.

The following theorem shows however that as long as we are only interested in the quotient $(T\times E)/G$,
 we can always replace $T$, $E$ and $G$ without changing the quotient so that the action of $G$ on $T\times E$ is diagonal and free in codimension one.

\begin{theorem} \label{thm:quotient}
Let $E$ be an elliptic curve, $T$ a smooth projective variety, and let $G$ be a finite group which acts on $T\times E$ and $T$ such that the projection $T\times E\to T$ is $G$-equivariant. 

Then there is an elliptic curve $E'$, a normal projective variety $T'$ and a finite group $G'$ which acts on $T'\times E'$ and $T'$, such that the projection $T'\times E'\to T'$ is $G'$-equivariant, with the following properties:

\begin{enumerate} 
\item There are compatible isomorphisms  \label{item:thm:quotient:1}
$$
(T'\times E')/G'\cong (T\times E)/G \ \ \text{and}\ \ 
 T'/G'\cong T/G .
$$   
\item The action of $G'$ on $T'\times E'$ is diagonal, i.e.\ $G'\subset \Aut(T')\times \Aut(E')$, and faithful on each factor of $T'\times E'$.\label{item:thm:quotient:2}
\item  There is a normal subgroup $G_0'\subset G'$ whose index is $\leq 4$ or $6$, such that the action of $G'_0$ on $T'\times E'$ is free.\label{item:thm:quotient:3}
\end{enumerate}  
\end{theorem}

\begin{remark}
While $T$ is assumed to be smooth in the above theorem, we cannot guarantee that $T'$ will be smooth.
Indeed, if the $G$-action on $T\times E$ is diagonal but trivial on $E$, then the theorem produces $T'=T/G$ which has quotient singularities.
\end{remark}
 
\begin{proof}[Proof of Theorem \ref{thm:quotient}]
Since $T\times E\to T$ is $G$-equivariant, the action of an element $g\in G$ on $T\times E$ is of the form
$$
g\cdot (t,e)=(g\cdot t,g(t)\cdot e)
$$
where $(g,t)\mapsto g\cdot t$ denotes the $G$-action on $T$ and where 
$
t\mapsto g(t)
$
yields a morphism $T\to \Aut(E)
$ with corresponding action $(g(t),e)\mapsto g(t)\cdot e$ on $E$.
Since $E$ is an elliptic curve, there is a short exact sequence
$$
0\longrightarrow E\longrightarrow \Aut(E)\stackrel{\alpha}\longrightarrow \Aut(E,0)\longrightarrow 0 ,
$$
where $E\subset \Aut(E)$ acts on $E$ via translation and $\Aut(E,0)$ denotes the image of the natural map $\Aut(E)\to \Aut(H^1(X,\Z))$, which can be identified with the subgroup of automorphisms of $E$ that fix the origin.
In particular, $\Aut(E,0)$ is a cyclic group of order $2,4$ or $6$.
Since $T$ is irreducible and $\Aut(E,0)$ is discrete, $\alpha(g(t))\in \Aut(E,0)$ does not depend on $t$.
Hence there is a well-defined group homomorphism
\begin{align} \label{eq:GtoAutE}
G\longrightarrow \Aut(E,0),\ \ g\mapsto \alpha(g(t)).
\end{align}
Let $G_0\subset G$ be the kernel of the above group homomorphism.
Then $G_0\subset G$ is a normal subgroup of index $\leq 4$ or $6$, because $\Aut(E,0)$ is cyclic of order $2,4$ or $6$, and we consider
$$
Y_0:=(T\times E)/G_0,
$$
which is a finite cover of $Y:=(T\times E)/G$ of degree $\leq 4$ or $6$.

\medskip

\textbf{Step 1.}
Consider the projection 
$$
p:Y_0=(T\times E)/G_0\longrightarrow S_0:=T/G_0.
$$
Then the monodromy action of $\pi_1(S_0)$ on $H^1(p^{-1}(s), \Q)$ for a closed point $s\in S_0$ is trivial, and $R^1p_\ast \Q$ is a trivial local system on $S_0$.

\begin{proof}
Since $G_0$ is in the kernel of (\ref{eq:GtoAutE}), any fibre of
$
p:Y_0\to S_0
$
is either smooth or a multiple of a smooth elliptic curve.
Hence, $R^1p_\ast\Q$ is a local system  (which above the
multiple fibres can be checked via topological base change).
The monodromy is trivial because $\pi_1(S_0)$ is an extension of $G_0$ by $\pi_1(T)$, where $\pi_1(T)$ acts trivially on $E$ and $G_0$ acts trivially on $H^1(E,\Q)$, as it acts via the translation of points on $E$.  
This concludes step 1.
\end{proof}

\textbf{Step 2.}
There is an elliptic curve $F$ with an action of $G/G_0$ and a $G/G_0$-equivariant morphism $h:Y_0\to F$ which restricts to finite \'etale maps on the fibres of $p:Y_0\to S_0$.

\begin{proof}
Since $T$ is smooth, $Y_0$ has only quotient singularities. 
Since quotient singularities are rational \cite[Proposition 1]{vieweg}, $Y_0$ has rational singularities.
Hence, there is a well-defined Albanese morphism $a:Y_0\to \Alb(Y_0)$, induced by the Albanese morphism of a resolution of $Y_0$ \cite[Proposition 2.3]{Re83}, and we may consider the dual abelian variety
$$
\Alb(Y_0)^\vee:=\Pic^0(\Alb(Y_0)) .
$$ 
Similarly, $S_0:=T/G_0$ has rational singularities and so $\Alb(S_0)$ is well-defined and we consider its dual
$$
\Alb(S_0)^\vee:=\Pic^0(\Alb(S_0)) .
$$ 

The natural map $p:Y_0\to S_0$ induces a morphism $p_\ast:\Alb(Y_0)\to \Alb(S_0)$ of abelian varieties and via duality a morphism  $p^\ast:\Alb(S_0)^\vee\to \Alb(Y_0)^\vee$. 
Note that $G/G_0$ acts on $S_0$ and $Y_0$ in a compatible way and so  $p^\ast:\Alb(S_0)^\vee\to \Alb(Y_0)^\vee$ is $G/G_0$-equivariant, where we define the $G/G_0$-action on line bundles $L$ on $\Alb(Y_0)$ (resp.\ $\Alb(S_0)$) by 
$$
g\cdot L:=(g^{-1})^\ast L,
$$
where $g^{-1}:\Alb(Y_0)\to \Alb(Y_0)$ (resp.\ $g^{-1}:\Alb(S_0)\to \Alb(S_0)$) denotes the group homomorphism that is induced by the action of $g^{-1}$ on $Y_0$ (resp.\ on $S_0$). 

We claim that
$$
F:=\Alb(Y_0)^\vee/p^\ast \Alb(S_0)^\vee
$$ 
is an elliptic curve.
To this end, let $Y'_0\to Y_0$ and $S'_0\to S_0$ be resolutions such that $p:Y_0\to S_0$ induces a morphism $p':Y'_0\to S'_0$. 
Since $Y_0$ and $S_0$ have rational singularities,
$$
F=\Alb(Y_0)^\vee/p^\ast \Alb(S_0)^\vee=\Alb(Y'_0)^\vee/(p')^\ast \Alb(S'_0)^\vee .
$$

By construction, $F$ is an abelian variety and it remains to show that its dimension is one, i.e. it suffices to prove that $b_1(Y'_0)-b_1(S'_0)=2$.
For this let $C$ be a general fibre of $p':Y_0'\to S_0'$, which is an elliptic curve.
Consider the Albanese map $a:Y_0'\to \Alb(Y_0')$.
Then $a(C)\subset \Alb(Y_0')$ is either a point or an elliptic curve.
In fact, it must be an elliptic curve, because the monodromy action on $H^1(C,\Q)$ is trivial by step 1 and so the global invariant cycle theorem implies that each class of $H^1(C,\Q)$ extends to $Y'_0$.
The composition $Y_0'\to  \Alb(Y_0')\to \Alb(Y_0')/a(C)$ contracts a general fibre of $p'$.
Since $p'$ has connected fibres, the above composition must factor through $p'$ and hence through the composition of $p'$ with the Albanese map of $S_0'$.
This shows that $\Alb(S'_0)$ is isogeneous to $\Alb(Y_0')/a(C)$.
Since $a(C)$ is an elliptic curve, we get $b_1(Y'_0)-b_1(S'_0)=2$, as claimed. 
We have thus shown that $F$ as above is indeed an elliptic curve.

The choice of a $G/G_0$-invariant ample divisor $D$ on $\Alb(Y_0)$ induces an isogeny 
$$
\phi:\Alb(Y_0)\longrightarrow\Alb(Y_0)^\vee , \ \ x\mapsto D-D_x, 
$$
where $D_x=D+x$ denotes the translate of $D$ by $x$. 
By our definition of the action of $G/G_0$ on $\Alb(Y_0)^\vee$, we find that for $g\in G/G_0$,
$$
g\cdot (D-D_x)=g(D)-g(D)_{g(x)}=D-D_{g(x)},
$$
because $D$ is $G/G_0$-invariant and the action of $g$ yields a group homomorphism $g:\Alb(Y_0)\to \Alb(Y_0)$.
Hence, $\phi$ is $G/G_0$-equivariant.

The morphism
$$
h:Y_0\longrightarrow F,
$$
given as composition
\begin{align} \label{eq:composition}
Y_0\stackrel{a}\longrightarrow \Alb(Y_0)\stackrel{\phi}\longrightarrow \Alb(Y_0)^\vee \twoheadrightarrow F= \Alb(Y_0)^\vee/p^\ast \Alb(S_0)^\vee 
\end{align}
restricts to finite \'etale covers on the fibres of $p$.
(Indeed, the image of any fibre of $p$  in $\Alb(Y_0)$ is an elliptic curve that maps to a point in $\Alb(S_0)$ and so it maps dually nonconstantly through the isogeny $\phi$ and the quotient map $\Alb(Y_0)^\vee \twoheadrightarrow F$ to $F$.)
Since $\phi$ is $G/G_0$-equivariant, all morphisms apart from the Albanese morphism $a$ in the above composition (\ref{eq:composition}) are $G/G_0$-equivariant.
Moreover, the Albanese morphism is $G/G_0$-equivariant up to the translation of a point, which depends on the base point $y_0\in Y_0$ that we implicitly chose in the definition of the Albanese map.
More precisely, for any $g\in G/G_0$ and any $y\in Y_0$,
$$
g\cdot a(y)-a(g\cdot y)=\int_{g\cdot y_0}^{g\cdot y}-\int_{y_0}^{g\cdot y}=\int_{g\cdot y_0}^{y_0}
$$
depends only on $g$ and the base point $y_0$ but not on $y$.
Hence,
\begin{align}\label{eq:g*h(y)}
g\cdot h(y)=h(g\cdot y)+t_g ,
\end{align}
where $t_g\in F$ does not depend on $y\in Y_0$.

Note that $\Aut(E,0)$ acts faithfully on $H^1(E,\Q)$.  
Hence, $G/G_0\subset \Aut(E,0)$ acts faithfully on $H^1(E,\Q)$.
Since $F$ is isogeneous to $E$, $H^1(E,\Q)\cong H^1(F,\Q)$ and the action of $G/G_0$ on $F$ is faithful as well.

Recall that $G/G_0$ is cyclic of order at most six and let $g\in G/G_0$ be a generator.
We may without loss of generality assume that $G/G_0\neq \{1 \}$ and so $g$ is not the identity.
Since $G/G_0$ acts faithfully on $F$, the morphism $g-\id:F\to F$ is surjective and so there is an element $s\in F$ with $g\cdot s-s=-t_g$.
We then define
$$
h':Y_0\longrightarrow F,\ \ y\mapsto h(y)+s .
$$
By (\ref{eq:g*h(y)}),
$$
g\cdot h'(y)=g\cdot h(y)+g\cdot s=h(g\cdot y)+t_g+g\cdot s=h'(g\cdot y) ,
$$
holds for the generator $g$ of the cyclic group $G/G_0$, and so it holds in fact for all $g\in G/G_0$.
Hence, up to replacing $h$ by $h'$ (which essentially amounts to the choice of a different base point for the Albanese morphism), we may assume that $h:Y_0\to F$ is $G/G_0$-equivariant. 
This concludes step 2.
\end{proof}

Let $S'\subset h^{-1}(x)$ be a connected component of a general fibre of $h:Y_0\to F$. 
By Bertini's theorem for normality, $S'$ is normal, hence in particular integral, because it is connected by definition.
Since $h$ restricts to finite \'etale maps on the fibres of $p$, $p|_{S'}:S'\to S_0$ is finite with ramification induced by the multiple fibres of $p$. 
In particular, $S'$ is a multi-section of $p:Y_0\to S_0$.

\medskip

\textbf{Step  3.}
The normalization $Y'$ of $Y_0\times_{S_0}S'$ has the following properties:
\begin{enumerate}
\item the  natural map \label{item:step3:1}
$
\tau:Y'\longrightarrow Y_0 
$
is quasi-\'etale;
\item $Y'$ has only quotient singularities and hence rational singularities;\label{item:step3:2}
\item there is an elliptic curve $E'$ and an isomorphism over $S'$:\label{item:step3:3}
$
Y'\cong S'\times E' .
$ 
\end{enumerate}

\begin{proof}
Since $S'$ is normal, a local computation shows that $\tau$ is \'etale in codimension one,  see e.g.\ \cite[Lemma 5.11]{HS19}.
This proves (\ref{item:step3:1}).

To prove (\ref{item:step3:2}), note that $Y_0=(T\times E)/G_0$ has quotient singularities by construction.
That is, if we put $X_0:=T\times E$, then $X_0$ is smooth and there is a finite Galois cover $\epsilon:X_0\to Y_0$.
Consider the fibre product $X':= X_0\times_{Y_0}Y'$, which fits into a diagram 
$$
\xymatrix{
X' \ar[d]_{\epsilon'}\ar[r]^{\tau'} & X_0\ar[d]^{\epsilon}\\
Y'\ar[r]_{\tau}& Y_0 .
}
$$
Since $\epsilon$ is finite and $\tau$ is finite quasi-\'etale, $\tau'$ is finite quasi-\'etale.
Since $X_0$ is smooth, and quasi-\'etale maps are \'etale over the smooth locus (see Corollary \ref{cor:quasi-etale=etale}), $\tau'$ is in fact \'etale.
Hence, $X'$ is smooth.
Since $\epsilon$ is Galois, so is $\epsilon'$ and so $Y'$ has quotient singularities.
This proves item (\ref{item:step3:2}) because quotient singularities are rational, see \cite[Proposition 1]{vieweg}.

It remains to show (\ref{item:step3:3}).
For this, we prove first that it suffices to show that there is an elliptic curve $E'$ such that $Y'$ is birational to $S'\times E'$ over $S'$.
Indeed, such a birational map induces a rational map $Y'\dashrightarrow E'$, given as composition
$$
Y'\stackrel{\sim}\dashrightarrow S'\times E'\stackrel{\pr_2}\longrightarrow E'.
$$
Since $Y'$ has rational singularities by item (\ref{item:step3:2}) proven above and $E'$ is an elliptic curve, this rational map is in fact a morphism $Y'\to E'$.
Together with the natural morphism $Y'\to S'$, this induces a birational morphism $Y'\to S'\times E'$ which is finite, because any fibre of $Y_0\to S_0$ is (a multiple of) an elliptic curve and $Y'\to Y_0\times_{S_0}S'$ is finite.
Hence,  $Y'\to S'\times E'$ is an isomorphism by Zariski's main theorem, because $S'\times E'$ is normal.

It remains to show that  there is an elliptic curve $E'$ such that $Y'$ is birational to $S'\times E'$ over $S'$.
For this we take a non-empty open subset $U'\subset S'$, such that $p':Y'\to S'$ is smooth over $U'$.
 (Since $Y'$ is normal and $\dim S'=\dim Y'-1$, such $U'$ exists.)
Note that $p':Y'\to S'$ has a natural rational section, given by the fact that $Y'$ is birational to $Y_0\times_{S_0}S'$ and $S'$ is a multi-section of $Y_0\to S_0$. 
Hence, after shrinking $U'$ if necessary, we may assume that the smooth elliptic fibration $(p')^{-1}(U')\to U'$ admits a regular section.
Moreover, $(p')^{-1}(U')\to U'$ has trivial monodromy by step 1 and so it follows from the existence of a fine moduli space of elliptic curves with level structure that 
\begin{align} \label{eq:thm3.1-step3}
(p')^{-1}(U')\cong U'\times E'
\end{align} 
for some elliptic curve $E'$.
That is, $p':Y'\to S'$ and $\pr_1:S'\times E'\to S'$ are isomorphic over $U'$, as we want.
This concludes the proof of item (\ref{item:step3:3}) and hence finishes step 3.
\end{proof}

By \cite[Theorem 3.7]{GKP16}, there is a normal variety $S''$ and a finite Galois cover $S''\longrightarrow T/G$ with Galois group $G''$, which factors through the composition $S'\to S_0\to T/G$.
Let 
$$
Y'':=Y'\times_{S'}S'' .
$$ 
By item (\ref{item:step3:3}) in step 3 there is an isomorphism 
\begin{align} \label{eq:Y''=S''xE'}
Y''\cong S''\times E'
\end{align}
which is compatible with the natural projections to $S''$ on both sides.

\medskip

\textbf{Step 4.} 
The natural 
commutative diagram
\begin{align} \label{diagr:step4}
\xymatrix{
Y''\ar[r]\ar[d] &Y'\ar[r]\ar[d] &  Y_0 \ar[r]\ar[d] & Y:=(T\times E)/G\ar[d]\\
S''\ar[r] & S'\ar[r] &S_0\ar[r] & S:=T/G 
}
\end{align}
has the property that all its squares become Cartesian after base change to a non-empty open subset $U$ of $S$.
Moreover, all varieties in the above diagram are normal.

\begin{proof} 
Since $Y''\cong S''\times E'$ and $S''$ is normal, $Y''$ is normal as well.
Normality of $Y'$ follows from its definition and $S'$ is normal by Bertini's theorem for normality, as noted in the paragraph above step 3.
Moreover, $S_0=T/G_0$, $S=T/G$, $Y_0=(T\times E)/G_0$ and $Y=(T\times E)/G$ are normal as quotients of a normal (in fact smooth) variety by the action of a finite group. 
Hence, all varieties in (\ref{diagr:step4}) are normal. 
It thus remains to prove the first claim in step 4.

Since $Y''=Y'\times_{S'}S''$, the left most square in (\ref{diagr:step4}) is Cartesian.
Next, the normalization morphism $Y'\to Y_0\times_{S_0}S'$ is an isomorphism over a non-empty smooth open subset of $S_0$ 
 and so the middle square is Cartesian after restriction to a non-empty open subset.
Finally, by the universal property of fibre products, there is a finite morphism $Y_0\to Y\times_{S}S_0$ of degree one, which must be the normalization because $Y_0=(T\times E)/G_0$ is normal.
In particular, $Y_0$ and $Y\times_{S}S_0$ are isomorphic over a non-empty open subset of $S_0$.

Altogether, it follows that in (\ref{diagr:step4}) all inner squares, and hence by a basic property of fibre products in fact all squares  become Cartesian after base change to a non-empty open subset $U$ of $S$.
This concludes step 4.
\end{proof}


Since $S''\to S$ is Galois with Galois group $G''$, Corollary \ref{cor:Zariski} implies that $S''\to S_0$ is Galois as well and we denote its Galois group by $G''_0\subset G''$. 

\medskip

\textbf{Step 5.}
The finite morphisms $Y''\to Y$ and $Y''\to Y_0$ are Galois with groups $G''$ and $G''_0$, respectively.
Moreover, $G''_0\subset G''$ is a normal subgroup and there is an isomorphism $G/G_0\cong G''/G''_0$ such that the two natural actions of these groups on $Y_0\cong Y''/G''_0$ coincide.
\begin{proof}
By step 4, all varieties in (\ref{diagr:step4}) are normal and all squares become Cartesian after restriction to a non-empty open subset $U\subset S$.
In particular, after restiction to $U\subset S$, the morphisms  $Y''\to Y$ and $Y''\to Y_0$ are Galois with group $G''$ and $G''_0$, respectively, because $G''$ and $G''_0$ are the Galois groups of $S''\to S$ and $S''\to S_0$, respectively. 
But then Theorem \ref{thm:Zariski} implies
that $Y''\to Y$ and $Y''\to Y_0$ are Galois with group $G''$ and $G''_0$, respectively, as we want.

By construction, $Y_0\to Y$ is Galois with Galois group $G/G_0$.
That is, there is a natural isomorphism $\Aut(Y_0/Y)\cong G/G_0$, where $\Aut(Y_0/Y)$ denotes the subgroup of automorphisms of $Y_0$ that lie over the identity of $Y$.

We aim to construct a surjective group homomorphism
\begin{align} \label{eq:chi}
\chi: G''\longrightarrow \Aut(Y_0/Y)
\end{align}
with kernel $G''_0$.
For this, we fix a general point $y_0''\in Y''$ and for any point $y''\in Y''$, we denote its image in $Y_0=Y''/G''_0$ by $[y'']$.
For a group element $g\in G''$, we then have
$$
[g''\cdot y_0''] =\varphi [y_0'']\in Y''/G''_0
$$
for some $\varphi \in \Aut(Y_0/Y)$, because $Y_0\to Y$ is Galois.
Note that $\varphi$ is unique because $y''_0$ is general 
 and so $\Aut(Y_0/Y)$ acts faithfully transitively on the fibre of $Y_0\to Y$ above the image of $y''_0\in Y''$ in $Y\cong Y''/G''$.
We may thus define 
$$
\chi(g''):=\varphi\in  \Aut(Y_0/Y).
$$
Note the corresponding map $\chi:G''\to  \Aut(Y_0/Y)$ is surjective, because for a given element $\varphi\in \Aut(Y_0/Y)$, the point $\varphi[y_0'']$ is always of the form $[g''\cdot y_0''] $ for some $g''\in G''$, because $\varphi$ is an automorphism of $Y_0$ which lies over the identity of $Y\cong Y''/G''$.

By construction,
\begin{align} \label{eq:g''cdot-y''}
[g''\cdot y'']=\chi(g'') [y'']\in Y''/G''_0
\end{align}
holds for $y''=y_0''$ and in fact for all $y''\in Y''$ in a small analytic neighbourhood $U$ of $y_0''$ such the quotient map $Y''\to Y''/G''$ splits into a disjoint union of $|G''|$ copies of $U$ above the image $U/G''\subset Y''/G''$ of $U$ in $Y''/G''$.
Consider then the two maps $Y''\to Y''_0$, given by $y''\mapsto [g''\cdot y'']$ and  $y''\mapsto \varphi[y'']$, respectively.
The locus where these  two maps coincide is a Zariski closed subset of $Y''$ which contains a small analytic neighbourhood of $y_0''$.
Since $Y''$ is irreducible, we deduce that (\ref{eq:g''cdot-y''}) holds
for all $y''\in Y''$.

We now claim that $\chi$ is a group homomorphism.
For this, let $g_1'',g_2''\in G''$.
Then we have
$$
\chi(g_1''g_2'')[y_0'']=[g_1''g_2''y_0'']=\chi(g_1'')[g_2''y_0'']=\chi(g_1'')\chi(g_2'')[y_0''],
$$ 
where we used that (\ref{eq:g''cdot-y''}) holds for all $y''\in Y''$.
Since  $\Aut(Y_0/Y)$ acts faithfully transitively on the fibre of $Y_0\to Y$ above the image of $y''_0\in Y''$ in $Y\cong Y''/G''$, we deduce that
$$
\chi(g_1''g_2'')=\chi(g_1'')\chi(g_2'')
$$
for all $g_1'',g_2''\in G''$ and so $\chi$ is a group homomorphism.

As we have seen above, $\chi$ is surjective.
Clearly, $G''_0\subset \ker(\chi)$ and so $\chi$ induces a surjection
$$
G''/G''_0 \longrightarrow \Aut(Y_0/Y).
$$
On the other hand, $G''/G''_0$ and $ \Aut(Y_0/Y)$ have the same number of elements, because $Y_0\cong Y''/G''_0\to Y\cong Y''/G''$ is Galois.
Hence, $\ker(\chi)=G''_0$ and so the group homomorphism in (\ref{eq:chi}) is surjective with kernel $G''_0$, as we want.
Since $\Aut(Y_0/Y)\cong G/G_0$, $\chi$ induces in particular a natural isomorphism
$$
G''/G''_0\cong G/G_0
$$
such that the two natural actions of these groups on   $(S''\times E')/G''_0\cong Y_0$ coincide.
This concludes step 5.
\end{proof}

We consider the composition
$$
h'':S''\times E'\cong Y''\stackrel{\pi_0}\longrightarrow (S''\times E')/G''_0\cong Y_0\stackrel{h}\longrightarrow F ,
$$
where $\pi_0$ denotes the quotient map.

\medskip

\textbf{Step 6.}
There is a finite \'etale morphism $\epsilon:E'\to F$ with
$
h''=\epsilon\circ \pr_2 .
$

\begin{proof}
Let $\tilde S''$ be a resolution of singularities of $S''$. 
This induces a resolution $\tilde Y''\cong \tilde S''\times E'$ of $Y''\cong S''\times E'$, see (\ref{eq:Y''=S''xE'}).
The morphism $h''$ induces a morphism $\tilde h'':\tilde Y''\to F$ and it suffices to show that this morphism factors through the second projection $\pr_2:\tilde S''\times E'\to E'$  (the claim that the map is \'etale is then automatic because $E'$ and $F$ are both elliptic curves).
By the universal property of the Albanese morphism, $\tilde h''$ factors as the composition of the Albanese morphism 
$$
\tilde Y''\longrightarrow \Alb(\tilde Y'')\cong \Alb(\tilde S'')\times E'
$$
and a morphism
$$
\phi: \Alb(\tilde S'')\times E'\longrightarrow F .
$$
Since $\phi$ is a morphism between abelian varieties, it is a group homomorphism up to  translation by a point.
This implies that it suffices to show that $\phi$ contracts $\Alb(S'')\times \{0\}$ to a point.
But for this it suffices to show that $\tilde h''$ contracts $\tilde S''\times \{0\}$ to a point.
Or equivalently, $h''$ contracts $S''\times \{0\}$ to a point.
By construction, $S''\times \{0\}$ maps via $Y''\to Y_0$ to $S'\subset h^{-1}(x)$ and so it is clear that $h''(S''\times \{0\})=x$ is a point.
This finishes the proof of step 6.
\end{proof}

By step 5, $G''$ acts on $Y''$ and this action is by construction compatible with the natural $G''$-action on $S''$. 
By (\ref{eq:Y''=S''xE'}), $Y''\cong S''\times E'$ such that $Y''\to S''$ corresponds to the first projection 
and so $G''$ acts naturally on $S''$ and $S''\times E'$ in compatible ways.
Moreover, step 5 implies that
$$
(S''\times E')/G''\cong Y=(T\times E)/G
$$
and $S''/G''\cong T/G$.
The action of an element $g''\in G''$ on $(s'',e')\in S''\times E'$ is thus of the form 
\begin{align} \label{eq:g''(s'')}
g''\cdot (s'',e')=(g''\cdot s'',g''(s'')\cdot e'),
\end{align}
for some morphism $S''\to \Aut(E')$.

\medskip

\textbf{Step 7.}
The action of $G''$ on $S''\times E'$ is diagonal, i.e.\ for any $g''\in G''$, the element $g''(s'')\in \Aut(E')$ in (\ref{eq:g''(s'')}) does not depend on $s''$.

\begin{proof}
Recall the $G/G_0$-equivariant morphism $h:Y_0\to F$ from step 2.
By step 5, $G''/G''_0\cong G/G_0$ and so $h$ is equivariant with respect to an action of $G''/G''_0$.
In particular,  the above $G/G_0$-action on $F$ induces an action of $G''$ on $F$ whose restriction to $G''_0$ is trivial.
Since $h'':Y''\to F$ is given by $h''=h\circ \pi_0$, it factors through the projection $\pi_0:Y''\to Y_0$  and so $h''$ is equivariant with respect to the natural actions of $G''$ on $Y''$ and $F$, respectively.
By (\ref{eq:Y''=S''xE'}), $Y''\cong S''\times E'$ and by step 6,  $h''=\epsilon\circ \pr_2$ for a finite map $\epsilon:E'\to F$.

Explicitly, for $g''\in G''$ and $(s'',e')\in S''\times E'$, (\ref{eq:g''(s'')}) then implies 
$$
h''(g''\cdot (s'',e'))=h''((g''\cdot s'',g''(s'')\cdot e') )=\epsilon(g''(s'')\cdot e' ).
$$
Since $h''$ is $G''$-equivariant, we also have
$$
h''(g''\cdot (s'',e'))=g''\cdot h'(s'',e')=g''\cdot \epsilon(e').
$$
Hence,
$$
\epsilon(g''(s'')\cdot e' )=g''\cdot \epsilon(e') 
$$
for all $s''\in S''$ and $e'\in E'$.
Since $\epsilon$ is finite, this is only possible if $g''(s'')$ in (\ref{eq:g''(s'')})  does not depend on $s''$.
Hence, the action of $G''$ on $S''\times E'$ is diagonal, as we want.
\end{proof}

By step 7, $G''\subset \Aut(S'')\times \Aut(E')$.
Since $S''\to S$ is Galois with group $G''$, we find that $G''$ acts faithfully on $S''$, i.e.\ the natural map $G''\to \Aut(S'')$ is injective.
A priori, $G''$ might not act faithfully on $E'$ and we denote the kernel of the natural group homomorphism $G''\to \Aut(E')$ by $H$.
Then $H\subset G''$ is a normal subgroup and we put 
$$
G':=G''/H\ \ \text{and}\ \ T':=S''/H.
$$
By construction, $G'$ acts on $T'\times E'$ and $T'$ with quotients
$$
(T'\times E')/G'\cong Y= (T\times E)/G \ \text{and}\ \ T'/G'\cong  T/G ,
$$
such that the natural maps
$$
(T'\times E')/G'\to T'/G'\ \ \text{and}\ \ (T\times E)/G\to T/G
$$
are identified with each other.
This proves item (\ref{item:thm:quotient:1}) in Theorem \ref{thm:quotient}.

Also, $G'$ acts diagonally and faithfully on each factor of $T'\times E'$ and so its action on $T'\times E'$ is free in codimension one, i.e.\ $T'\times E'\to Y$ is quasi-\'etale.
This proves item (\ref{item:thm:quotient:2}) in Theorem \ref{thm:quotient}.

Finally, if we put $G'_0:=G''_0/H$, then $G'/G'_0\cong G''/G''_0$.
It thus follows from step 5 and the construction of $G_0\subset G$ in (\ref{eq:GtoAutE}) that $G'_0$ is a normal subgroup of $G'$ of index $\leq 4$ or $6$, and $(T'\times E')/G'_0\cong Y_0 $.
Since $p:Y_0\to S_0$ has trivial monodromy by step 1, the action of $G'_0$ on $E'$ is given by translation by torsion points.
Since $G'$ acts faithfully on $E'$, it follows that $G'_0$ acts freely on $E'$ and hence freely on $T'\times E'$.
This proves item (\ref{item:thm:quotient:3}), which concludes the theorem. 
\end{proof}

\begin{remark}
A crucial point in the above argument is the existence of the $G/G_0$-equivariant morphism $h:Y_0\to F$ from step 2.
This uses in an essential way that $G/G_0$ is a cyclic group, which in turn relies on the classification of the automorphism groups of elliptic curves.
In particular, the above proof does not generalize to the situation where $E$ is an abelian variety of higher dimension or a smooth projective curve of higher genus, so that the quotient of the automorphism group $\Aut(E)$ by the subgroup of those automorphisms that act trivial on cohomology, is not cyclic.
It is conceivable that also the result of Theorem \ref{thm:quotient} does not generalize to this more general setting.
\end{remark}
 
\section{Minimal models with $\kappa=n-1$ and generically isotrivial Iitaka fibration}

In this section we use Theorem \ref{thm:quotient} to prove Theorem \ref{thm:quotient-type}, stated in the introduction.

\begin{proof}[Proof of Theorem \ref{thm:quotient-type}]
Let $X$ be a minimal model of dimension $n$ and Kodaira dimension $n-1$.
Then $X$ is a good minimal model by Theorem \ref{thm:good-minmod-kappa=n-1} and so the Iitaka fibration $f:X\to S$ is a morphism with 
\begin{align} \label{eq:KX=f*A}
K_X \sim_{\Q} f^\ast A
\end{align} 
for some ample $\Q$-divisor $A$ on $S$.

Since (\ref{item:quotient-type:2}) $\Rightarrow$ (\ref{item:quotient-type:1})  in Theorem \ref{thm:quotient-type} is obvious, it suffices to prove the converse implication.
For this, we assume that $f$ is generically isotrivial.
We then proceed in several steps.
 
\textbf{Step 1.}
There is a finite group $G'$ and a non-empty open subset $U\subset S$, such that $X_U:=f^{-1}(U)\to U$ is given by
$$
X_U\cong (U'\times E')/G' \to U'/G'\cong U,
$$
where $E'$ is an elliptic curve and $U'$ is a $G'$-variety with $U'/G'\cong U$ whose action on $U'$ lifts to $U'\times E'$ making the first projection $G'$-equivariant.

\begin{proof}
By assumptions, $f$ is generically isotrivial with typical fibre an elliptic curve $E'$.
Let $U\subset S$ be a sufficiently small Zariski open non-empty subset, such that $U$ and $f|_{X_U}:X_U\to U$ are smooth.
Then $R^1f_\ast \Z|_U$ is a local system.
Moreover, $X_U$ is smooth and $f|_{X_U}$ is a proper morphism between complex manifolds which is locally isotrivial.
Hence, $f|_{X_U}$ is an analytic fibre bundle by \cite{FG}.
Since the identity component of $\Aut(E')$ acts trivially on $H^1(E',\Z)$, this implies that the monodromy of $R^1f_\ast \Z|_U$ is finite.
That is, for any fixed base point $u\in U$, the image of $\pi_1(U,u)\to \Aut(H^1(X_u,\Z))=\GL_2(\Z)$ is a finite group.
Hence, there is an irreducible variety $U'$ and a finite Galois \'etale covering $U'\to U$ with Galois group $G'$, such that the base change $X_{U'}:=X\times_SU'$ is an elliptic fibre bundle over $U'$ with trivial monodromy.
Since $X$ is projective, $X_{U'}\to U'$ admits a rational multi-section.
Hence, up to shrinking $U$ and replacing $U'\to U$ by another Galois \'etale covering (where we use Corollary \ref{cor:Zariski}), we may assume that $f':X_{U'}\to U'$ admits a section, cf.\ \cite[Section 6.9]{Kol93} for a similar argument.
Note also that $R^1f'_\ast\Z$ is a trivial local system on $U'$ by construction.
The existence of a fine moduli space for elliptic curves with level structure thus shows
\begin{align} \label{eq:X_U2}
X_{U'}\cong U'\times E',
\end{align}
because we know that any fibre of $X_{U'}\to U'$ is isomorphic to $E'$.
Moreover, since $X_{U'}\to X_U$ is Galois with Galois group $G'$, there is a natural action of $G'$ on $U'\times E'$ with quotient $X_U$.
This concludes step 1.
\end{proof}

\textbf{Step 2.}
There is an elliptic curve $E$, a normal projective variety $T$, and a finite group $G\subset \Aut(T)\times \Aut(E)$ whose action on $T\times E$ is diagonal and faithful on each factor, such that:
\begin{itemize}
\item there is a normal subgroup $G_0\subset G$ of index $\leq 4$ or $6$ whose action on $T\times E$ is free;
\item  $Y=(T\times E)/G$ fits into a commutative diagram
\begin{align} \label{diag-step2}
\xymatrix{
X \ar[d]^f \ar@{-->}[r]^-\varphi  & Y=(T\times E)/G \ar[d]^{g}\\
S     & \ar[l]^h T/G ,
}
\end{align}
where $\varphi$ is a birational map, $h$ is a birational morphism, and $g$ is the canonical morphism induced by the projection $\pr_1:T\times E\to T$.  
\end{itemize} 

\begin{proof}
By step 1, there is an open subset $U\subset S$, such that  $X_{U}:=f^{-1}(U)$ is a quotient
$$
X_U=(U'\times E')/G' .
$$ 
By equivariant resolution of singularities \cite{AW97}, there is a smooth projective $G'$-variety $T'$, with a $G'$-equivariant birational map to $U'$, such that $T'/G'$ admits a birational morphism to $S$.
The action of an element $g'\in G'$ on $U'\times E'$ is of the form
$$
g'\cdot (u',e')=(g'\cdot u',g'(u')\cdot e')
$$
for a morphism $g':U'\to\Aut(E')$.
Since $T'$ is smooth and each connected component of $\Aut(E')$ is isomorphic to $E'$, the induced rational map $T'\dashrightarrow \Aut(E')$ must be a morphism $T'\to \Aut(E')$ and so the action of $G'$ on $U'\times E'$ extends to a $G'$-action on $T'\times E'$ such that $T'\times E'\to T'$ is $G'$-equivariant.
The claim in step 2 then follows by applying Theorem \ref{thm:quotient} to the action of $G'$ on $T'\times E'$.
\end{proof}

\textbf{Step 3.}
We may in step 2 assume that:
\begin{itemize}
\item  $T$ is a canonical model over $S$; in particular, $T$ has only canonical singularities and $K_T$ is ample over $S$.
\item $K_Y \sim_{\Q} g^\ast B$ for an $h$-ample $\Q$-divisor $B$ on $T/G$, where $g$ and $h$ are as in (\ref{diag-step2}).
\end{itemize}

\begin{proof}
Replacing $T$ by a $G$-equivariant resolution of singularities, we may assume that $T$ is smooth.
The natural map $T\to S$ is $G$-equivariant, where we take the trivial action of $G$ on $S$.
The relative canonical model $T^c$ of $T$ over $S$ exists because $T$ is smooth and $T\to S$ is generically finite, see \cite{BCHM}.
Since $T\to S$ is $G$-equivariant, $T^c$ inherits a natural action of $G$, see Section \ref{subsec:G-MMP}.
This action induces a diagonal action on $T^c\times E$ which coincides birationally with the given diagonal action of $G$ on $T\times E$.
Hence, up to replacing $T$ by $T^c$, we may assume that $T\to S$ is a relative canonical model and so $K_T$ is ample over $S$ and $T$ has only canonical singularities.

Since $G$ acts diagonally and faithfully on each factor of $T\times E$, the quotient map $\xi:T\times E\to Y$ is quasi-\'etale and so 
$$
\xi^\ast K_Y \sim K_{T\times E}  \sim \pr_1^\ast K_T,
$$
because $E$ is an elliptic curve.
Note that $K_{T\times E}$ is $\Q$-Cartier and so is $K_Y$ by \cite[Proposition 5.20]{kollar-mori}. 
Since $K_T$ is $G$-invariant, this shows that $K_Y\sim_\Q g^\ast B$ for a $\Q$-Cartier $\Q$-divisor $B$ on $T/G$ whose pullback to $T$ coincides with $K_T$. 
Since $K_T$ is relatively ample over $S$ and $T\to T/G$ is finite, $B$ must be $h$-ample, as we want.
\end{proof}

\textbf{Step 4.}
The rational map $\varphi:X\dashrightarrow Y$ in (\ref{diag-step2}) does not extract any divisor. 
In particular, $\varphi_\ast K_X  \sim K_Y$.
 
\begin{proof}
For a contradiction, assume that $\varphi$ extracts some divisors  $E_i\subset Y$, i.e.\ the $E_i$ are the $\varphi^{-1}$-exceptional divisors.
Since $X$ is terminal, we find
\begin{align} \label{eq:K_Y}
K_Y  \sim_{\Q} \varphi_\ast K_X+\sum a_iE_i
\end{align}
for some $a_i>0$.
Recall that $\varphi$ is an isomorphism over an open subset of $S$ by step 1 (and the argument in step 2, resp.\ in Theorem \ref{thm:quotient}). 
Since $g$ is equi-dimensional, this implies that each $E_i$ is of the form $E_i=(D_i\times E)/G$ for some $G$-invariant divisor $D_i\subset T$ which is contracted via $p:T\to S$.
 In fact, the $D_i$ are exactly the divisors on $T$ that are contracted by the natural map $p:T\to S$. 

Since $T\times E\to Y$ is quasi-\'etale, $K_Y$ pulls back to $K_{T\times E}  \sim_{\Q} \pr_1^\ast K_T$ and so (\ref{eq:K_Y}) implies
\begin{align} \label{eq:K_T}
K_T \sim_{\Q} p^\ast A+\sum a_iD_i ,
\end{align}
where $D_i$ and $a_i$ are as above and $A$ is the ample $\Q$-divisor on $S$ with $K_X \sim_{\Q} f^\ast A$ from (\ref{eq:KX=f*A}).
Let 
$$
T\stackrel{\rho}\longrightarrow S'\stackrel{p'}\longrightarrow S
$$ 
be the Stein factorization of $p:T\to S$. 
Then $\rho$ is birational and the exceptional divisors of $\rho$ are exactly given by the divisors $D_i$. 

Since $K_T$ is ample over $S$, we deduce from (\ref{eq:K_T}) that $\sum a_iD_i $ is $\Q$-Cartier (because $p^\ast A$ and $K_T$ are $\Q$-Cartier) and ample over $S'$.
Note also that the pushforward of  $-\sum a_iD_i $  to $S'$ is trivial  and hence effective.
 But applying \cite[Lemma 3.39]{kollar-mori} to the birational morphism $\rho$, we get that $-\sum a_iD_i $ must be effective,  
which is a contradiction, because $a_i>0$.
\end{proof}

\textbf{Step 5.}
In (\ref{diag-step2}), $h$ is an isomorphism,
$Y$ has only canonical singularities, $K_Y  \sim_{\Q} g^\ast h^\ast A$ and $K_T$ is ample, given by the pullback of $A$ via $T\to T/G\cong S$.

\begin{proof}
By step 4, $
\varphi_\ast K_X  \sim K_Y 
$.
Since  $K_X \sim_{\Q} f^\ast A$ and $K_Y \sim_{\Q} g^\ast B$, this implies
$$
g^\ast h^\ast A \sim_{\Q} g^\ast B.
$$
Since $B$ is $h$-ample, this implies that $h$ must be finite and so it is an isomorphism, because it is birational by construction.
Using $h$ to identify $S$ with $T/G$, we find $K_Y\sim_{\Q} g^\ast B  \sim_{\Q} g^\ast  h^\ast A$ and so $K_X \sim_{\Q}\varphi^\ast K_Y$ (see Section \ref{subsec:convention} for the definition of $\varphi^*$).
Let $\tilde X\to X$ be a resolution of the birational map $\varphi:X\dashrightarrow Y$.
Then $\tilde X$ is a common resolution of $X$ and $Y$ and so we can use $\tilde X$ to compute the discrepancies of $X$ and of $Y$. 
Since $X$ is terminal, the only possible exceptional divisors with nonpositive discrepancy for $\tilde{X}\to Y$ are coming from exceptional divisors of $\varphi$, which in fact have discrepancy $0$ because $K_X  \sim_{\Q} \varphi^\ast K_Y$.
This implies that $Y$ has only canonical singularities.
Finally, since $T\times E\to Y$ is quasi-\'etale by step 2, $K_Y  \sim_{\Q} g^\ast  h^\ast A$ pulls back to $K_{T\times E}  = \pr_1^\ast K_T$.
This shows that  $K_T$ is $\Q$-linearly equivalent to the pullback of $A$ via the finite morphism $T\to T/G\cong S$.
Since $A$ is ample, it follows that $K_T$ is ample as well. 
This concludes step 5. 
\end{proof}

\textbf{Step 6.}
The birational map $\varphi:X\dashrightarrow Y$ in (\ref{diag-step2}) factors as
$$
X \stackrel{\phi}\dashrightarrow X^+\stackrel{\tau}\longrightarrow Y=(T\times E)/G,
$$
where $\phi$ is a sequence of flops and $\tau$ is a terminalization of $Y$. 

\begin{proof}
Since $Y$ is canonical by step 5, there is a terminalization $\tau:X^+\longrightarrow Y$ with $\tau^\ast K_Y = K_{X^+}$, see Section \ref{subsec:terminalization}.
Since $K_Y\sim_{\Q} g^\ast B$ and $h$ is an isomorphism by step 5, $K_{X^+}$ is nef over $S$.
Hence, $X$ and $X^+$ are birational minimal models over $S$ and so they are connected by a sequence of flops, see \cite{kawamata-flops}. 
\end{proof}

The proof of Theorem \ref{thm:quotient-type} follows immediately from steps 2, 5 and 6 above.
\end{proof}

\begin{corollary} \label{cor:quotient-type}
Assume in the notation of Theorem \ref{thm:quotient-type} that for any codimension one point $s\in S^{(1)}$, the fibre of $f$ above $s$ is an irreducible curve.
Then the action of $G$ on $T\times E$ is free in codimension two.
\end{corollary}
\begin{proof}  
Assume for a contradiction that the action of $G$ on $T\times E$ is not free in codimension two. 
Since $G$ acts diagonally and faithfully on each factor, there is a nontrivial element $g\in G$ that fixes a codimension two set of the form $D\times \{e\}$ for some $e\in E$ and some prime divisor $D\subset T$.
Since $G$ acts freely in codimension one on $T\times E$, the fixed points of nontrivial elements of $G$ correspond to singular points of the quotient $Y=(T\times E)/G$ (see e.g.\ \cite[p.\ 296, Corollary]{Fuj74}) and so the image 
$$
\overline {D\times \{e\}} \subset Y
$$
of $D\times \{e\}$ lies in the singular locus of $Y$.
Moreover, the image of this codimension two locus via $g:Y\to T/G$ is a divisor $\overline D\subset T/G$. 
Since terminal varieties are smooth in codimension two, the terminalization $\tau:X^+\to Y$ in Theorem \ref{thm:quotient-type} resolves the singularity at the generic point of $\overline {D\times \{e\}}$.
Since $\phi$ is an isomorphism in codimension one, this implies that $X$ contains a prime divisor $R\subset X$ with $ \tau(\phi(R))= \overline {D\times \{e\}}$.

Firstly, notice that the fibres of $g:Y\to S$ are smooth elliptic or rational curves. 
Secondly, $\tau$ has connected fibres by Zariski's main theorem because $Y$ is normal. 
Thirdly, all fibres of $g\circ\tau$ over codimension one points are curves, since all fibres $f$ over codimension one points are curves, and $\phi$ is a composition of flops, hence an isomorphism in codimension one.
Thus the terminalization $\tau$, which resolves the singularities at generic point of $\overline {D\times \{e\}}$, will add more components in the fibres of $g$ above the generic point of $\overline {D}\cong f(R)$.
But then  
the fibre of $f$ above the generic point of $f(R)\subset S$ must be reducible, which contradicts our assumption because $f(R)\cong \overline D$ is a divisor on $S\cong T/G$.
This concludes the corollary.
\end{proof}

\begin{corollary} \label{cor:T=canonical}
In the notation of Theorem \ref{thm:quotient-type}, $T$ has canonical singularities.
\end{corollary}
\begin{proof}
Since $G$ acts diagonally and faithfully on each factor, the quotient map $T\times E\to Y$ is quasi-\'etale.
Since $Y$ is canonical by Theorem \ref{thm:quotient-type}, this implies that $T\times E$ is canonical, see \cite[Proposition 5.20.(3)]{kollar-mori}.
But then it is clear that $T$ is canonical as well.
\end{proof}

\section{Proof of main results}
\label{sec:proof}

\subsection{Proof of Theorem \ref{thm:kappa=n-1}}

\begin{lemma} \label{lem:Iitaka-flops}
Let $X$ and $X^+$ be good minimal models that are connected by a sequence of flops $\phi:X\dashrightarrow X^+$.
Let $f:X\to S$ and $f:X^+\to S^+$ be the respective Iitaka fibrations.
Then there is an isomorphism $\psi:S\stackrel{\sim}\longrightarrow S^+$ such that $\psi\circ f=f^+\circ \phi$.
\end{lemma}
\begin{proof}
It suffices to prove the lemma in the case where $\phi$ is a single flop.
In this case, let $g:X\to Z$ and $g^+:X^+\to Z$ be the corresponding flopping contractions.
Since $K_X$ is trivial on all curves that are contracted by $g$, the cone theorem \cite[Theorem 3.7.(4)]{kollar-mori} implies that $K_X=g^\ast L$ for some $\Q$-Cartier divisor $L$ on $Z$.
Since $g$ and $g^+$ are small contractions, $K_{X^+}=(g^+)^\ast L$.
This implies that the Iitaka fibration of $X$ and $X^+$ factor both through the Iitaka fibration of the $\Q$-Cartier divisor $L$ on $Z$, and so the statement is clear.
This proves the lemma.
\end{proof}

\begin{proof}[Proof of Theorem \ref{thm:kappa=n-1}]
Let $X$ be a minimal model of dimension $n$ and Kodaira dimension $n-1$.
By Theorem \ref{thm:good-minmod-kappa=n-1}, $X$ is a good minimal model and so the Iitaka fibration $f:X\to S$ is a morphism which equips $X$ with the structure of an elliptic fibre space. 

Let us first assume that $c_1^{n-2}c_2(X)=0$.
Then, by Corollary \ref{cor:grassi}, the smooth fibres of $f$ have constant $j$-invariants and the fibre of $f$ over any codimension one point of $S$ is an irreducible curve (it is either smooth or a multiple of a smooth elliptic curve).
Hence,  Theorem \ref{thm:quotient-type}, Corollary \ref{cor:quotient-type} and Corollary \ref{cor:T=canonical} imply that $X$ is birational to a quotient $Y=(T\times E)/G$ with canonical singularities, where $T$ is a projective variety with canonical singularities and ample canonical class, $E$ is an elliptic curve and $G\subset \Aut(T)\times \Aut(E)$ acts diagonally, faithfully on each factor and freely in codimension two on $T\times E$.
This proves one direction in Theorem \ref{thm:kappa=n-1}.

Conversely, assume that $X$ is birational to a quotient $Y=(T\times E)/G$, where $T$ is a projective variety with ample canonical class, and $G\subset \Aut(T)\times \Aut(E)$ acts diagonally, faithfully on each factor and freely in codimension two on $T\times E$ such that $Y$ is canonical. 
We then need to show that $c_1^{n-2}c_2(X)=0$.

Since $T\times E\to Y$ is quasi-\'etale and $K_T$ is ample, $K_Y$ is nef.
Since $Y$ is canonical, any terminalization $\tau:X^+\to Y$ of $Y$ (see Section \ref{subsec:terminalization}) is a minimal model that is birational to $X$.
Hence the natural birational map $\phi:X\dashrightarrow X^+$ is a sequence of flops by \cite{kawamata-flops}.
By Lemma \ref{lem:Iitaka-flops}, the Iitaka fibration $f:X\to S$ is given by the composition of $\phi$ with the Iitaka fibration $f^+:X^+\to S^+$ of $X^+$. 
(Note that both Iitaka fibrations are actual morphisms by Theorem \ref{thm:good-minmod-kappa=n-1}.)
Since $\tau$ is crepant, the Iitaka fibration of $X^+$ is given by the composition of $\tau$ and the Iitaka fibration of $Y$.
Since $T\times E\to Y$ is quasi-\'etale and $K_T$ is ample, the Iitaka fibration of $Y$ is a morphism, given by the natural map $g:Y\to T/G$, induced by the first projection $T\times E\to T$.
Altogether, we conclude that there is an isomorphism $S\cong T/G$ such that the Iitaka fibration $f:X\to S$ is given by the composition of the birational map $\tau\circ \phi:X\dashrightarrow Y$ and the Iitaka fibration $g:Y\to T/G$ of $Y$.
That is, we arrive at a commutative diagram
\begin{align} \label{eq:diag}
\xymatrix{
X \ar[dr]_f \ar@{-->}[r]^-\phi &X^+\ar[r]^-\tau & Y=(T\times E)/G \ar[dl]^{g}\\
& S  \cong T/G & .
}
\end{align}

Since $f$ is the Iitaka fibration of $X$, there is a very ample line bundle $A$ on $S$ such that $K_X$ is a linearly equivalent to a rational multiple of $f^\ast A$.
Let $C\subset S$ be the intersection of $n-2$ general elements of the linear series $|A|$.
By Lemma \ref{lem:grassi}, $c_1^{n-2}c_2(X)=0$ is then equivalent to showing that $Z:=f^{-1}(C)$ is a minimal elliptic surface over $C$ such that all singular fibres of $Z\to C$ are multiples of smooth elliptic curves.

Since $K_X$ is nef and the normal bundle of $Z$ in $X$ is a direct sum of nef line bundles, $K_Z$ is nef.
By Bertini's theorem and because $X$ is terminal, hence smooth in codimension two,  $Z$ is smooth.
In particular, $Z\to C$ is a minimal elliptic surface.

Since $C\subset S$ is the intersection of general hyperplanes, (\ref{eq:diag}) shows that $Z$ is birational to $g^{-1}(C)\subset Y$.
If $\tilde C$ denotes the preimage of $C\subset T/G$ via the quotient map $T\to T/G$, then $G$ acts on $\tilde C$ and 
$$
g^{-1}(C)=(\tilde C\times E)/G .
$$
Since $T$ is normal and $C\subset S\cong T/G$ is the intersection of general hyperplanes, $\tilde C$ is smooth by Bertini's theorem.
Since $G$ acts diagonally, faithfully on each factor and freely in codimension two on $T\times E$, it also follows that $G$ acts freely on $\tilde C\times E$. 
(Once again, this uses that $\tilde C\subset T$ is in general position, as $C\subset S\cong T/G$ is the intersection of general hyperplanes.) 
In particular, the quotient map $\tilde C\times E\to g^{-1}(C)$ is \'etale and we conclude that $g^{-1}(C)$ is a smooth projective minimal surface, birational to $Z$.
Minimal models of surfaces are unique and so $Z\cong g^{-1}(C)$.
Since $\tilde C\times E\to g^{-1}(C)$ is \'etale, all singular fibres of $Z\cong g^{-1}(C)\to C$ must be multiples of smooth elliptic curves, as we want.  
This concludes the proof of Theorem \ref{thm:kappa=n-1}.
 \end{proof}
 
\subsection{Proof of Corollary \ref{cor:factor-phi}}
\begin{proof}[Proof of Corollary \ref{cor:factor-phi}]
We use the notation of Theorem \ref{thm:kappa=n-1}.
Since $G$ acts diagonally and faithfully on each factor of $T\times E$, the quotient map $T\times E\to Y=(T\times E)/G$ is quasi-\'etale.
Hence, $K_{Y}$ pulls back to $K_{T\times E}  \sim_{\Q} \pr_1^\ast K_T$, which is nef by item (\ref{item:thm1:KT=ample}) in Theorem \ref{thm:kappa=n-1}.
This implies that $K_Y$ is nef.
Since $Y$ is canonical by item (\ref{item:thm1:Y=sm-in-codim2}) in Theorem \ref{thm:kappa=n-1}, there is a terminalization $X^+\to Y$, see Section \ref{subsec:terminalization}.
In particular, $X^+$ is a minimal model, birational to $Y$.
Hence, $X$ and $X^+$ are birational minimal models and so they are connected by a sequence of flops \cite{kawamata-flops}.
This proves item (\ref{item:cor:factor-phi:1}) of the corollary.

Since $T\times E\to Y$ is quasi-\'etale and $Y$ has canonical singularities, $T$ has canonical singularities as well (see Corollary \ref{cor:T=canonical}).
Since $K_T$ is ample by item (\ref{item:thm1:KT=ample}), it follows that $T$ is its own canonical model.
This proves item (\ref{item:cor:factor-phi:3}) of Corollary \ref{cor:factor-phi}.

The Iitaka fibration $f:X\to S$ of $X$ factors through any sequence of flops as well as through the terminalization $X^+\to Y$, as in each step only $K$-trivial curves are flopped or contracted, see Lemma \ref{lem:Iitaka-flops}.
Item (\ref{item:cor:factor-phi:2}) thus follows from (\ref{item:cor:factor-phi:1}) if we can show that 
the Iitaka fibtration of $Y$ is given by the projection $Y=(T\times E)/G\to T/G$.
The Iitaka fibration of $T\times E$ is clearly given by the first projection $T\times E\to T$.
Since $T\times E\to Y$ is quasi-\'etale, $K_Y$ pullsback to $K_{T\times E}$.
Hence, the Iitaka fibration of $Y$ is induced by those pluricanonical forms on $T\times E$ that are $G$-invariant.
This shows that the Iitaka fibration of $Y$ is given by $Y\to T/G$, as claimed.
This concludes item (\ref{item:cor:factor-phi:2}) in Corollary \ref{cor:factor-phi}.

It remains to show that the triple $(T,E,G)$ depends up to isomorphism only on the birational equivalence class of $X$.
To show this, let $(T',E',G')$ be another triple as in Theorem \ref{thm:kappa=n-1} such that $X$ is birational to $Y'=(T'\times E')/G'$.
We then need to show that there are compatible isomorphisms $T\cong T'$, $E\cong E'$ and $G\cong G'$.

As we have seen above, the Iitaka fibration of $X$ then factors as $X\dashrightarrow Y\to T/G$ and also as $X\dashrightarrow Y'\to T'/G'$. 
Hence, there is an isomorphism $T/G\cong T'/G'$ and a birational map $\psi:Y\dashrightarrow Y'$  such that the following diagram commutes
$$
\xymatrix{
Y=(T\times E)/G \ar[dr]_g \ar@{-->}[rr]^{\psi}  &  & Y'=(T'\times E')/G' \ar[dl]^{g'}\\
&T/G  \cong T'/G' & ,
}
$$
where $g$ and $g'$ are induced by the corresponding projections onto the first factors.
The birational map $\psi$ induces a birational map on the generic fibres of $g$ and $g'$.
These are elliptic curves and so $\psi$ must restrict to an isomorphism between the generic fibres of $g$ and $g'$.
In particular, $\psi$ is an isomorphism above a nonempty open subset of $T/G\cong T'/G'$.

Since $G$ and $G'$ act diagonally and faithfully on each factor, the general fibres of $g$ and $g'$ are  given by $E$ and $E'$, respectively.
By what we have seen above, $\psi$ thus induces an isomorphism $E\cong E'$.

Next, let $G_0\subset G$ and $G'_0\subset G'$ be the normal subgroups given as kernels of the natural compositions
$$
G\hookrightarrow \Aut(E)\longrightarrow \GL(H^1(E,\Z))\ \ \text{and}\ \ G'\hookrightarrow \Aut(E')\longrightarrow \GL(H^1(E',\Z)),
$$
respectively.
The sheaves $R^1g_\ast \Z$ and $R^1g'_\ast \Z$ are generically local systems on $T/G\cong T'/G'$.
 Also, $R^1g_\ast \Z$ and $R^1g'_\ast \Z$  are generically isomorphic  to each other, as they both coincide generically with $R^1f_\ast \Z$, where $f:X\to S\cong T/G$ is the Iitaka fibration.
The quotients $G/G_0$ and $G'/G'_0$ can then both be identified with the monodromy group of the generic local system $R^1f_\ast \Z$ and so we get a canonical isomorphism $ G/G_0\cong G'/G'_0$.
Moreover, the covers
$$
T/G_0\to T/G\cong T'/G'\ \ \text{and}\ \ T'/G'_0\to T'/G'\cong T/G\
$$
are generically isomorphic, because they coincide generically to the  lowest degree cover that trivializes the monodromy of $R^1f_\ast \Z$.
Since $T/G_0$ and $T'/G_0'$ are both normal, Theorem \ref{thm:Zariski} then implies that the above generic isomorphism extends to a unique isomorphism $T/G_0\cong T'/G_0'$.
Note that there are natural isomorphisms of fibre products 
$$
Y\times_{T/G} T/G_0\cong (T\times E)/G_0\ \ \text{and}\ \ Y'\times_{T'/G'} T'/G'_0\cong (T'\times E')/G'_0
$$
Since $T/G_0\cong T'/G_0'$, the birational map $\psi$ thus induces a canonical birational map
$$
  (T\times E)/G_0 \dashrightarrow  (T'\times E')/G'_0 .
$$  

Since $G_0\subset G$ is intrinsically determined as the kernel of $G\to \Aut(E)\to \GL(H^1(E,\Z))$, the extension $0\to G_0\to G\to G/G_0\to 0$ is determined by $G_0$ together with the image of $G/G_0$ inside $\GL(H^1(E,\Z))$.
As we have already seen that there is an isomorphism $E\cong E'$ such that $G/G_0$ and $G'/G_0'$ correspond to the same subgroup of $\GL(H^1(E,\Z))\cong \GL(H^1(E',\Z))$, we see that $G$ (resp.\ $G'$) can completely be recovered from $G_0$ (resp.\ $G_0'$).
Hence, up to replacing $G$ by $G_0$ and $G'$ by $G_0'$, we may from now on assume that $G$ acts trivially on $H^1(E,\Z)$ and $G'$ acts trivially on $H^1(E',\Z)$.

Since $ Y$ and $Y'$ have canonical and hence rational singularities by item (\ref{item:thm1:Y=sm-in-codim2}), their Albanese varieties are well-defined and isomorphic: $\Alb(Y)\cong \Alb(Y')$, cf.\ \cite[Proposition 2.3]{Re83}.
As we have seen in step 2 of the proof of Theorem \ref{thm:quotient}, this implies that 
$$
F:=\Alb(Y)^\vee/g^\ast \Alb(T/G)^\vee
$$
is an elliptic curve (here we use in an essential way the above reduction to the case where $G=G_0$ acts trivially on $H^1(E,\Z)$).
We fix an isogeny 
$$
\Alb(Y)\cong \Alb(Y')\longrightarrow \Alb(Y)^\vee  \cong \Alb(Y')^\vee .
$$
As in step 2 of the proof of Theorem \ref{thm:quotient}, we then get maps
$$
h:Y\longrightarrow F\ \ \text{and}\ \ h':Y'\longrightarrow F.
$$
Moreover, both maps are compatible with the birational map $\psi$ between $Y$ and $Y'$.
This implies that $\psi$ induces a birational map between the general fibres of $h$ and $h'$.
As mentioned above, step 2 of Theorem \ref{thm:quotient} implies $\dim\Alb(Y)-\dim\Alb(T/G)=1$ (here we again use the above reduction that $G=G_0$ acts trivially on $H^1(E,\Z)$). 
Hence the surjective morphism $Y=(T\times E)/G \to T/G\times E/G$ induces an isogeny from $\Alb(Y)$ to $\Alb(T/G)\times E/G$.
The same holds for $Y'$.
Hence, $\Alb(Y)$ and $\Alb(Y')$ are isogeneous to $\Alb(T/G)\times E/G$ and $\Alb(T'/G')\times E'/G$, respectively.
This implies that the Stein factorizations of $h$ and $h'$, coincides with the projections $Y\to E/G$ and $Y'\to E'/G'$, respectively.
The general fibres of $Y\to E/G$ and $Y'\to E'/G'$ are given by $T$ and $T'$, respectively. 
Hence, each connected component of a general fibre of $h$ is isomorphic to $T$, while each connected component of a general fibre of $h'$ is isomorphic to $T'$.
As we have seen above, $\psi$ induces a birational map between the general fibres of $h$ and $h'$, respectively. 
Hence, $\psi$ induces a birational map $T\dashrightarrow T'$.
By item (\ref{item:cor:factor-phi:3}), $T$ and $T'$ are their own canonical models and so the above birational map must be an isomorphism: $T\cong T'$.
Since $G$ resp.\ $G'$ is the Galois group of the cover $T\to T/G$, resp.\ $T'\to T'/G'$, the isomorphisms $T\cong T'$ and $T/G\cong T'/G'$ also yield an isomorphism $G\cong G'$.
It follows easily from our construction that this isomorphism of groups is induced via restriction from the isomorphism $\Aut(T)\times \Aut(E)\cong \Aut(T')\times \Aut(E')$ that is induced by $T\cong T'$ and $E\cong E'$.
Altogether, we find that $T$, $E$ and $G\subset \Aut(T)\times \Aut(E)$ is uniquely determined up to isomorphism, as claimed in item (\ref{item:cor:factor-phi:4}) of Corollary \ref{cor:factor-phi}.
This concludes the proof of the corollary.
\end{proof}

\subsection{Examples of non-diagonal group actions} \label{subsec:example}

\begin{proof}[Proof of the claim in Example \ref{ex:CxE}]
Let $E$ be an elliptic curve with an automorphism $\varphi$ of order $3$ which fixes  the origin of $E$.
Consider the quotient map $\pi:E\to E/\varphi=\CP^1$ and let $\tau:\CP^1\to \CP^1$ be a non-trivial finite morphism, branched at general points of $\CP^1$.
We may then consider the fibre product $C:=E\times_{\CP^1} \CP^1$, giving rise to a Cartesian diagram
$$
\xymatrix{
C \ar[r]^-{\tau'}\ar[d]^{\pi'} & E\ar[d]^{\pi} \\
\CP^1\ar[r]^{\tau}&\CP^1.
}
$$
Then $C$ is a smooth projective curve of general type which comes with an automorphism $\varphi'$ of order three, induced by $\varphi$.
Since the above diagram is Cartesian, this automorphism has the property that
\begin{align} \label{eq:tau-varphi}
\varphi\circ \tau'=\tau'\circ \varphi' .
\end{align}
By the Hurwitz formula, $\pi$ is ramified at $3$ points (with ramification index $2$ at each point).
This shows that there is a $3$-torsion point $t\in E$ such that $\pi$ is not ramified at $t$.
We may then consider the automorphism
$$
\psi: C\times E\longrightarrow C\times E,\ \ (x,y)\mapsto (\varphi'(x),\tau'(x)+y-t).
$$
Since $\varphi'$ has order three, we have $\varphi'\circ \varphi'+\varphi'+\id=0$.
Moreover, $\tau'(\varphi'(x))=\varphi(\tau'(x))$ for all $x\in C$ by (\ref{eq:tau-varphi}).
Altogether, this easily implies that $\psi$ is an automorphism of order three.
Since $t$ is not a ramification point of $\pi$, it is not a fixed point of $\varphi$ and so one easily checks that $\psi$ has no fixed points.
Hence,
$$
X:=(C\times E)/\psi
$$
is a smooth projective surface with a finite \'etale cover $\epsilon:C\times E\to X$.
Since $C\times E$ is minimal of Kodaira dimension one, the same holds for $X$.
Moreover, 
$$
c_2(X)=\frac{1}{3}\epsilon^\ast c_2(X)=\frac{1}{3}c_2(C\times E)=0,
$$
which concludes the proof.
\end{proof}


\subsection{Classification of all groups that appear in Theorem \ref{thm:kappa=n-1}}

%
%
\begin{proof}[Proof of Corollary \ref{cor:classification-of-groups}]
The only if part follows from Theorem \ref{thm:kappa=n-1}, because $G\subset \Aut(T)\times \Aut(E)$ acts faithfully on each factor and so $G$ is a subgroup of $\Aut(E)$ for some elliptic curve $E$.
Conversely, let $G\subset \Aut(E)$ be a finite subgroup for some elliptic curve $E$.
Then the quotient $C:=E/G$ is a smooth projective curve of genus at most one which comes with a quotient map $\pi:E\to C$.
Let $f:C'\to C$ be a finite covering, branched at a positive number of points that are disjoint from the branch points of $\pi$.
We consider the fibre product $E':=C'\times_CE$, which sits in a commutative diagram
$$
\xymatrix{
E'\ar[r]^{f'}\ar[d]_{\pi'} & E\ar[d]^{\pi} \\
C'\ar[r]_{f}&C 
} 
$$
Here, $E'$ is a smooth projective curve and $f':E'\to E$ is branched at a finite number of points, so that $E'$ is a smooth projective curve of genus at least two.
We claim that there is a natural injective group homomorphism $\chi:G\to \Aut(E')$.
To see this, let $g\in G\subset \Aut(E)$ and consider the morphism $g\circ f':E'\to E$.
Together with $\pi'$, this induces by the universal property of the fibre product a unique isomorphism $\chi(g):E'\to E'$ with
\begin{align} \label{eq:cor:classification}
g\circ f'=f'\circ \chi(g).
\end{align}
The uniqueness statement shows that $\chi(g_1g_2)=\chi(g_1)\chi(g_2)$ because
$$
g_1\circ g_2 \circ f'=g_1\circ f'\circ \chi(g_2)=f'\circ \chi(g_1) \circ \chi(g_2)
$$
and so $\chi$ is a group homomorphism.
If $g\in \ker(\chi)$, then $g\circ f'=f'$ and so $g=\id$, because $f'$ is surjective.
This proves that $\chi:G\to \Aut(E')$ is an injective group homomorphism.
Hence, there is an injective action of $G$ on $E'$.

For a positive integer $n\geq 2$, we may then consider the diagonal action of $G$ on $T_n:=(E')^n$.
Since $E'$ is a smooth projective curve of general type, $T_n$ is a smooth projective variety with ample canonical class  $K_{T_n}=\otimes_{i=1}^n\pr_i^*K_{E'}$, where $\pr_i$ is the $i$-th projection.
The quotient
$$
Y_{n}:=(T_{n-1}\times E)/G
$$
by the diagonal action of $G$ on $T_{n-1}\times E$ has isolated singularities, because the $G$-action on $T_{n-1}\times E$ is free outside a finite set of points. 

We claim that $Y_{n}$ has terminal singularities as soon as $n> |G|$.
To see this, recall that $Y_n$ has isolated singularities and let $y\in Y_n$ be such an isolated singularity. 
Let  $(t,e)\in T_{n-1}\times E$ be a lift of $y$ and let $H\subset G$ be the maximal subgroup of elements that fix the point $(t,e)$.
Then $H\neq \{1\}$ and there is an embedding $H\hookrightarrow \GL_n(\C)$ such that  locally in the analytic topology, there is a neighbourhood of $y$ in $Y_n$ that is isomorphic to a neighbourhood of the origin in the quotient of $\C^n$ by $H\subset \GL_n(\C)$.
By construction, any non-trivial element $g\in G$ acts with at most finitely many fixed points on $T_{n-1}\times E$.
This implies that each non-trivial element $h\in H\subset \GL_n(\C)$ has the origin as its unique fixed point, i.e.\ the eigenvalues of $h$ are all different from $1$.
If $m$ denotes the order of $H$ and $\xi$ denotes a primitive $m$-th root of unity, then the eigenvalues of $h$ are of the form $\xi^{a_1},\dots ,\xi^{a_n}$ for some integers $0\leq a_i<m$.
Since none of the eigenvalues of $h$ is one, $a_i\geq 1$ for all $i$ and so
$$
\operatorname{age}(h):=\frac{1}{m}\cdot \sum_{i=1}^na_i\geq \frac{n}{m}.
$$ 
Since $n> |G|\geq |H|=m$, $\operatorname{age}(h)> 1$ and so $y\in Y_n$ is a terminal singularity by the Reid--Tai criterium, see e.g.\ \cite[Theorem 3.21]{kollar-singularities}. 

Altogether, this shows that $X:=Y_n=(T_{n-1}\times E)/G$ is a minimal model of dimension $n$ and of Kodaira dimension $n-1$.
By Theorem \ref{thm:kappa=n-1}, $c_1^{n-2}c_2(X)=0$ and so we see that the group $G$ appears in Theorem \ref{thm:kappa=n-1} when applied to $X$.
This concludes the proof of the corollary. 
\end{proof}

\subsection{Proof of Corollary \ref{cor:Cartier-index}}

\begin{lemma} \label{lem:Cartier-index}
Let $X$ be a minimal model and let $\phi:X\dashrightarrow X^+$ be a flop.
Then the Cartier index of $X$ and $X^+$ coincide, i.e.\ $m\cdot K_X$ is Cartier if and only if $mK_{X^+}$ is Cartier.
\end{lemma}
\begin{proof}
The flop $\phi$ fits into a diagram
$$
\xymatrix{
X \ar@{-->}[rr] \ar[dr]_f & & X^+ \ar[dl]^{f^+}\\
 &Z&,
}
$$
where $f:X\to Z$ is a flopping contraction, i.e.\ the contraction of a $K_X$-trivial ray that is small.

Assume that $m\cdot K_X$ is Cartier.
By the cone theorem \cite[Theorem 3.7.(4)]{kollar-mori}, $mK_X=f^\ast L$ for some line bundle $L$ on $Z$.
Since $\phi$ is an isomorphism in codimension one, $mK_{X^+}=(f^+)^\ast L$, and so this divisor is Cartier as well.
On the other hand, $\phi^{-1}:X^+\dashrightarrow X$ is also a flop (see \cite[Definition 6.10]{kollar-mori}) and so the same argument as above shows that $mK_X$ is Carrier if and only if $mK_{X^+}$ is Cartier.
Hence, the Cartier indices  of $X$ and $X^+$ coincide, as we want.
\end{proof}

\begin{proof}[Proof of Corollary \ref{cor:Cartier-index}]
Assume that the Iitaka fibration $f:X\to S$ is generically isotrivial.
By Theorem \ref{thm:quotient-type}, we get a commutative diagram
$$
\xymatrix{
X \ar[dr]_f \ar@{-->}[r]^-\phi &X^+\ar[r]^-\tau & Y=(T\times E)/G \ar[dl]^{g}\\
& S  \cong T/G & ,
}
$$
where $\phi$ is a composition of flops, $\tau$ is a terminalization, and $g$ is induced by the projection $T\times E\to T$. 
Moreover,  there is a normal subgroup $G_0\subset G$ of index $\leq 4$ or $6$ whose action on $T\times E$ is free. 

By Lemma \ref{lem:Cartier-index}, $X$ and $X^+$ have the same Cartier indices.
Also, $K_{X^+}=\tau^\ast K_Y$ and so it suffices to show that $Y$ has Cartier index $\leq 4$ or $6$.
To see this, note that $X$ is a threefold and so $T$ is a surface with canonical singularities, by Corollary \ref{cor:T=canonical}.
Hence, $T$ is Gorenstein, see e.g.\ \cite[Theorem 7.5.1]{Ish14} and so the same holds true for $T\times E$. 

Recall the normal subgroup $G_0\subset G$ of index $\leq 4$ or $6$ from Theorem \ref{thm:quotient-type} which acts freely on $T\times E$.
We claim that $Y_0:=(T\times E)/G_0$ is Gorenstein as well.
To see this, note that $T\times E\to Y_0$ is \'etale.
Since $T\times E$ is Gorenstein, this readily implies that $\mathcal O_{Y_0}(K_{Y_0})$ is locally free in the analytic category.
But then  $\mathcal O_{Y_0}(K_{Y_0})$ is locally free in the algebraic category by Serre's GAGA-principle, where we use that $Y_0$ is projective.
Hence, $Y_0$ is Gorenstein.
Since $G$ acts diagonally and faithfully on each factor of $T\times E$, the quotient map $T\times E\to Y=(T\times E)/G$ is quasi-\'etale.
Hence,
$$
Y_0=(T\times E)/G_0\longrightarrow Y=(T\times E)/G
$$
is quasi-\'etale as well.
Since $Y_0$ is Gorenstein (as we have seen above), it follows that the Gorenstein index of
$Y$
divides the order of $G/G_0$, see e.g.\ 
 \cite[Proposition 5.20]{kollar-mori}.
Since $G/G_0$ has order $\leq 4$ or $6$, $4K_Y$ or $6K_Y$ is Cartier.
As explained above, this implies that $4K_X$ or $6K_X$ is Cartier, which concludes the proof.
\end{proof}

Corollary \ref{cor:Cartier-index} has the following immediate consequence. 

\begin{corollary} \label{cor:c1c2>epsilon-constant-j}
Let $X$ be a minimal threefold of Kodaira dimension two.
Assume that the smooth fibres of the Iitaka fibration $f:X\to S$ have constant $j$-invariants.
Then 
$$
4c_1c_2(X)\in \Z \ \ \text{or}\ \ 6c_1c_2(X)\in \Z  .
$$  
\end{corollary}

\subsection{Proof of Theorem \ref{thm:c1c2>epsilon}}

\begin{proof}[Proof of Theorem \ref{thm:c1c2>epsilon}]
Let $X$ be a minimal threefold.
If $X$ is of general type, then the statement follows from the Bogomolov--Miyaoka--Yau inequality (\ref{eq:miyaoka-ineq-gen-type}) and the universal lower bound on $K_X^3$ from \cite{HM06,Tak06}.
If $X$ is Calabi-Yau, then $c_1^3(X)=0$ by the abundance conjecture \cite{Ka92}.
If $X$ has Kodaira dimension one, then the result follows from \cite[Theorem 1.5]{gra94} and if $X$ has Kodaira dimension two and the Iitaka fibration of $X$ is not generically isotrivial, then the statement follows from  \cite[Proposition 2.5]{gra94} (which in turn relies on results from \cite{Kol94}).
The remaining case of Kodaira dimension two and generically isotrivial Iitaka fibration follows from Corollary \ref{cor:c1c2>epsilon-constant-j}.
This concludes the proof.
\end{proof}

\section*{Acknowledgement} 
The second author thanks I.\ Cheltsov, S.\ Filipazzi, Th.\ Peternell and C.\ Shramov for useful discussions.
We thank the referees for their careful reading and for a question that lead us to discover the uniqueness statement in item (\ref{item:cor:factor-phi:4}) of Corollary \ref{cor:factor-phi}.
Both authors are supported by the DFG grant ``Topologische Eigenschaften von Algebraischen Variet\"aten'' (project nr.\ 416054549). 
The first author is also
supported by grant G097819N of Nero Budur from the Research Foundation Flanders.

\end{document}